\documentclass[10pt]{amsart}

\usepackage{graphicx}        % standard LaTeX graphics tool
\usepackage{multicol}        % used for the two-column index
\usepackage[bottom]{footmisc}% places footnotes at page bottom
\usepackage{amscd,amsmath,amssymb,amsfonts}
\usepackage[cmtip, all]{xy}

\newtheorem{theorem}{Theorem}[section]
\newtheorem{proposition}[theorem]{Proposition}
\newtheorem{lemma}[theorem]{Lemma}
\newtheorem{corollary}[theorem]{Corollary}

\theoremstyle{definition}
\newtheorem{definition}[theorem]{Definition}
\theoremstyle{remark}
\newtheorem{remark}[theorem]{Remark}

\newtheorem{ex}[theorem]{Example}

\numberwithin{equation}{section}

\newcommand{\red}{{\rm red}}
\newcommand{\codim}{{\rm codim}}

\newcommand{\Spec}{{\rm Spec\,}}
\newcommand{\sing}{{\rm sing}}
\newcommand{\Char}{{\rm char}}
\newcommand{\Tr}{{\text{Tr}}}

\newcommand{\supp}{{\rm supp}\,}
\newcommand{\0}{\emptyset}
\newcommand{\sC}{{\mathcal C}}
\newcommand{\sD}{{\mathcal D}}

\newcommand{\sJ}{{\mathcal J}}

\newcommand{\sL}{{\mathcal L}}
\newcommand{\sM}{{\mathcal M}}
\newcommand{\sN}{{\mathcal N}}
\newcommand{\sO}{{\mathcal O}}
\newcommand{\sP}{{\mathcal P}}

\newcommand{\sS}{{\mathcal S}}

\newcommand{\sU}{{\mathcal U}}

\newcommand{\A}{{\mathbb A}}

\newcommand{\C}{{\mathbb C}}

\renewcommand{\P}{{\mathbb P}}

\newcommand{\R}{{\mathbb R}}

\renewcommand{\det}{\operatorname{det}}
\newcommand{\Nm}{{\operatorname{Nm}}}

\newcommand{\id}{{\operatorname{\rm Id}}}

\newcommand{\<}{\langle}
\renewcommand{\>}{\rangle}

\renewcommand{\dim}{{\operatorname{\rm dim}}} 
 
\newcommand{\xr}[1]{\xrightarrow{#1}}

\newcommand{\del}{\partial}

\newcommand{\Sm}{{\mathbf{Sm}}}

\newcommand{\Sym}{{\operatorname{Sym}}}

\newcommand{\Sing}{{\operatorname{Sing}}}

\newcommand{\GW}{{\operatorname{GW}}} 
\newcommand{\sGW}{{\mathcal{GW}}}

\newcommand{\GL}{\operatorname{GL}}

\newcommand{\PGL}{\operatorname{PGL}}

\newcommand{\Wel}{{\operatorname{Wel}}}

\newcommand{\Res}{{\operatorname{Res}}}

\newcommand{\ev}{\text{\it ev}}

\newcommand{\ind}[1]{}

\newcommand{\inp}[1]{}

\newcommand{\res}{\text{res}}

\renewcommand{\setminus}{-}

\begin{document}
\setcounter{tocdepth}{1}

\title{Toward an algebraic theory of Welschinger invariants}

\author{Marc Levine}
\address{Fakult\"at Mathematik\\
Universit\"at Duisburg-Essen\\
Thea-Leymann-Str. 9\\
45127 Essen\\
Germany}
\email{marc.levine@uni-due.de}



\thanks{The author would like to thank the DFG for its support through the SFB Transregio 45 and the SPP 1786 ``Homotopy theory and algebraic geometry''}

\begin{abstract} Let $S$ be a smooth del Pezzo surface over a field $k$ of characteristic $\neq 2, 3$. We define an invariant in the Grothendieck-Witt ring $\GW(k)$ for ``counting'' rational curves  in a curve class $D$ of fixed positive degree  (with respect to the anti-canonical bundle $-K_S$)  and containing a collection of distinct closed points $\mathfrak{p}=\sum_ip_i$ of total degree $r:=-D\cdot K_S-1$ on $S$. This recovers Welschinger's invariant in case $k=\R$ by applying the signature map. The main result is that this quadratic invariant depends only on the $\A^1$-connected component containing $\mathfrak{p}$ in $\Sym^r(S)^0(k)$, where  $\Sym^r(S)^0$ is the open subscheme of $\Sym^r(S)$ parametrizing geometrically reduced 0-cycles. \end{abstract}
\maketitle

\tableofcontents

\section*{Introduction} 
One fascinating aspect of real Gromov-Witten invariants is the theory  of Welsch\-inger invariants in the setting of symplectic/almost complex geometry. This was  developed by Welschinger in \cite{Wel2, Wel1}, where he considers the real aspects of the algebraic enumerative problem of counting the number of maps $f:\C\P^1\to X$, $X$ a smooth projective rational surface over $\C$, with $f_*([\C\P^1])$ in a fixed homology class, and passing through a fixed configuration of $n$ points in $X$, with $n$ chosen so one expects a finite number of such maps. Welschinger transforms this into the setting of almost complex structures on a symplectic four manifold $X$ which has a real structure given by a involution mimicking complex conjugation on the $\C$-points of a real algebraic surface. This allows one to pose the problem of counting the real maps $f:S^2\to X$, with $S^2$ and $X$ given compatible almost complex structures, such that $f$ sends a fixed set of points $s_\bullet$ of $S^2$ to a fixed set $x_\bullet$ on $X$. Phrased in this way, the varying configuration of points on the algebraic surface $X$ and the projective line gets replaced by a fixed set of points $s_\bullet$, $x_\bullet$,  with the variation taking place in the almost complex structures. 

In any case, the number of real solutions to this problem turns out in general to vary with the choice of almost complex structures. However, Welschinger assigns a ``mass'' to a general map $f:S^2\to X$ as follows: A general $f$ will be unramified and the image $f(S^2)$ with have finitely many double points. Using the real structure, one can divide the double points into three classes: those occurring in complex conjugate pairs, the real points with two real branches and the real points with complex conjugate branches. Welschinger discards the complex conjugate pairs of double points, gives the real points with real branches the sign +1 and those with complex conjugate branches the sign -1. The product of these signs is the {\em mass} of the map $f$, and Welschinger's main result is that, if one fixed the real type of the configuration of points, that is, if one fixes the number of points in each real connected component of $X$, then the mass-weighted sum of maps $f$  is independent of the (suitably general) choice of almost complex structure. 

Itenberg, Kharlamov and Shustin have reformulated Welschinger's work in the setting of real algebraic geometry in their paper \cite{IKS}. In a somewhat different direction, Kass and Wickelgren \cite{KW1, KW2} have proven a series of results on quadratic refinements of invariants from enumerative geometry, which has strongly influenced this paper.  They have recently announced a quadratic refinement of the count of bi-nodal curves in a dimension two linear system of curves on $\P^2$ via a quadratic refinement of the iterative method of Kleiman and Piene \cite{KPi}.   In this paper we give an different approach, generalizing the Welschinger invariant for curves on del Pezzo surfaces from the case of real algebraic geometry to algebraic geometry over an arbitrary perfect field (of characteristic not equal to 2 or 3). For simplicity, we describe our approach in the case of $\P^2$.

We consider the moduli stack of stable maps of genus 0 curves $f:C\to \P^2$ with a reduced degree $n$ divisor $\mathfrak{d}$ on $C$ with $f_*([C])$ of degree $d$,  $\bar{\sM}_{0,n}^\Sigma(\P^2,d)$. Sending $\mathfrak{d}$ to the image 0-cycle $f_*(\mathfrak{d})$ defines a morphism of stacks
\[
\Pi^\Sigma:\bar{\sM}_{0,n}^\Sigma(\P^2,d)\to [\mathfrak{S}_n\backslash (\P^2)^n]
\]
We consider the open substack (actually subscheme) $(\Sym^n\P^2)^0$ of reduced degree $n$ 0-cycles on $\P^2$, and the restriction of $\Pi^\Sigma$ to the inverse image substack
\[
\mathfrak{z}:\bar{\sM}_{0,3d-1}^\Sigma(\P^2,d)^0\to (\Sym^{3d-1}\P^2)^0
\]

We construct for each unramified birational map $f:C\to \P^2_K$, with $C$ a smooth geometrically irreducible genus zero curve over a field $K$,  and $f(C)$ having only ordinary double points, an invariant $\Wel_K(f(C))$ in the Grothendieck-Witt ring $\GW(K)$. For $K=\R$, the signature of
$\Wel_K(f(C))$  is (up to a sign depending only on $d$) Welschinger's mass. For $\eta\in (\Sym^{3d-1}\P^2)^0$ such that $\Wel_{k(x)}(f_x(C_x))$ is defined  for all $x=(f_x:C_x\to \P^2_{k(x)})$ lying over $\eta$,   taking a sum of traces down to $k(\eta)$ over the $\Wel_{k(x)}(f_x(C_x))$ gives us an element $\Wel_d(\eta)\in \GW(k(\eta))$. One shows rather easily that this gives a section of the sheaf $\sGW$ of Grothendieck-Witt rings over some open subscheme of $(\Sym^{3d-1}\P^2)^0$. Our main result, completely analogous to Welschinger's invariance result, is that this construction actually defines a global section $\Wel_d\in H^0((\Sym^{3d-1}\P^2)^0, \sGW)$. In consequence, for $K$ a field, we have  
\[
\Wel_d(\mathfrak{d}_1)=\Wel_d(\mathfrak{d}_2) \in \GW(K)
\]
if we have  0-cycles $\mathfrak{d}_1, \mathfrak{d}_2$  in the same  $\A^1$-connected component of $(\Sym^{3d-1}\P^2)^0(K)$.

For $K=\R$, one can show that the $\A^1$-connected components of $(\Sym^{3d-1}\P^2)^0(\R)$ just depend on the number of real points in the given 0-cycle. The rank of $\Wel_d(\mathfrak{d})$ gives the count of number of curves over $\C$ and the signature of  of $\Wel_d(\mathfrak{d})$ gives Welschinger's invariant (up to the sign mentioned above), so we recover Welschinger's invariance result for curves on del Pezzo surfaces, when we restrict to actual complex structures, as opposed to almost complex structures. In particular, this recovers the results of Itenberg-Kharlamov-Shustin \cite{IKS1}. 

As pointed out to me by Jesse Kass, in the special case of  cubic curves in $\P^2$ passing through a degree eight 0-cycle, the Welschinger invariant   is just the sum of traces of the local $\A^1$-invariants at the singular points of the singular fibers of the resulting pencil. In this case, the Welschinger invariant is related to the Euler characteristic of $\P^2$ blown up along the base-locus of the pencil by the Riemann-Hurwitz formula proven in \cite[Theorem 12.2]{LevEnum}. As the base-locus is the given degree eight 0-cycle plus a single rational point, the resulting Euler characteristic is completely determined by the original degree eight 0-cycle; in the case of the ground field $k=\R$, the analogous fact was already noted in Welschinger's paper \cite{Wel2}. 

I would like to thank Kirsten Wickelgren and Jesse Kass for their interesting and helpful comments.  In particular, Jesse Kass pointed out that the local invariant defined here does {\em not} in general agree with the one defined by Shende and G\"ottsche in \cite{ShenGott} (or rather with the image of the Shende-G\"ottsche local invariant in $\GW(k)$ with respect to the motivic measure given by an appropriate version of the quadratic Euler characteristic). This is not to say that either of the two local invariants is incorrect but rather that there should be a relation that says the two global invariants do agree. 

The paper is organized as follows. We begin with recalling the moduli stack of $n$-pointed stable maps of a genus 0-curve to $\P^2$, and its quotient stack by the symmetric group permuting the points. This is the basic object of our study. In \S2 we define and study the Welschinger invariant for a curve in the plane, refining Welschinger's ``mass''. In \S 3 we set up our main result and give the proof, relying on results obtained in the following four sections. These are the technical heart of the paper, giving an analysis of the behavior of our quadratic Welschinger invariant at four different degenerate types of stable maps: curves whose image has a cusp, a tacnode or a triple point, or curves consisting of two irreducible components. In essence, we follow the arguments of Welschinger and Itenberg-Kharlamov-Shustin, we need only be a bit more careful. As in these earlier papers, the most subtle case is that of a cusp.

To fix ideas and keep the argument as simple as possible, we consider in the first seven sections the case of curves of degree $d$ in $\P^2$. In the final section, we explain how to generalize these results to arbitrary del Pezzo surfaces. 

We work throughout this paper over a base-field $k$ of characteristic $\neq 2, 3$.

\section{Stable maps of genus 0 curves}  Fix positive integers $d, n$. We have the moduli stack $\bar{\sM}_{0,n}(\P^2, d)$ of stable maps $f:C\to \P^2$ of $n$-pointed geometrically connected quasi-stable genus zero curves $C$ with $f^*(\sO_{\P^2}(1))$ an invertible sheaf on $C$ of degree $d$, with open substack $\sM_{0,n}(\P^2, d)$ of those $f:C\to \P^2$ with $C$ smooth. Letting $[C]$ denote the fundamental class of $C$ as a dimension 1 cycle on $C$, the degree condition on $f$ is the same as requiring that $f_*([C])$ is a degree $d$, dimension one cycle on $\P^2$.

$\sM_{0,n}(\P^2, d)$  is a proper smooth Artin stack over $k$ (see e.g. \cite[\S 1.3]{FP} for a treatment over $\C$ and \cite{AbramOort} for an arbitrary base scheme) and is irreducible by \cite{J, KP}.

The symmetric group $\mathfrak{S}_n$ acts on $\bar{\sM}_{0,n}(\P^2, d)$ by permuting the pointing, giving us the quotient stack $\bar{\sM}_{0,n}^\Sigma(\P^2, d):=[\mathfrak{S}_n\backslash \bar{\sM}_{0,n}(\P^2, d)]$ and the open substack
$\sM_{0,n}^\Sigma(\P^2, d):=[\mathfrak{S}_n\backslash \sM_{0,n}(\P^2, d)]$

 For a triple $(f, C, (p_1,\ldots, p_n))$ representing an $n$-pointed genus zero curve  over some base-scheme $S$ and map $f:C\to \P^2_S$, we have the $S$-point $(f(p_1),\ldots, f(p_n)):S\to (\P^2)^n$; sending $(f, C, (p_1,\ldots, p_n))$ to 
$(f(p_1),\ldots, f(p_n))$ defines a morphism 
\[
\Pi:\bar{\sM}_{0,n}(\P^2, d)\to (\P^2)^n, 
\]
equivariant with respect to the action of $\mathfrak{S}_n$ on $\bar{\sM}_{0,n}(\P^2, d)$ and the permutation action on $(\P^2)^n$. Thus $\Pi$ induces a proper map on the quotient stacks
\[
\Pi^\Sigma:\bar{\sM}_{0,n}^\Sigma(\P^2, d)\to [\mathfrak{S}_n\backslash (\P^2)^n]
\]

A degree $n$ 0-cycle $\mathfrak{d}$ on a   $K$-scheme $Y$ is {\em reduced} if $\mathfrak{d}_{\bar{K}}$ is  a sum of $n$ distinct $\bar{K}$-points. Equivalently, $\mathfrak{d}=\sum_{i=1}^rp_i$ with the $p_i$ distinct closed points of $Y$ such that $k(p_i)$ is a separable extension of $K$ for each $i$. We may consider $\bar{\sM}_{0,n}^\Sigma(\P^2, d)$ as the stack parametrizing pairs $(f:C\to \P^2, \mathfrak{d})$ where $f$ is a stable map of a genus zero semi-stable curve $C$ with $f_*([C])$ of degree $d$,  and $\mathfrak{d}$ is a reduced divisor on $C$ of degree $n$.

Let $(\P^2)^n_0\subset (\P^2)^n$ be the open complement of all diagonals in $(\P^2)^n$. As $\mathfrak{S}_n$ acts freely on the quasi-projective scheme $(\P^2)^n_0$, the quotient stack $[\mathfrak{S}_n\backslash (\P^2)^n_0]$ is the same as the quotient scheme $\mathfrak{S}_n\backslash (\P^2)^n_0$, which in turn is the open subscheme $\Sym^n(\P^2)^0$ of the symmetric power $\Sym^n \P^2$ parametrizing reduced degree $n$ 0-cycles on $\P^2$. We let  $\bar{\sM}_{0,n}^\Sigma(\P^2, d)^0\subset \bar{\sM}_{0,n}^\Sigma(\P^2, d)$ be the open substack $(\Pi^\Sigma)^{-1}(\Sym^n(\P^2)^0)$ and let 
\[
\mathfrak{z}:\bar{\sM}_{0,n}^\Sigma(\P^2, d)^0\to \Sym^n(\P^2)^0
\]
denote the restriction of $\Pi^\Sigma$. We may thus consider $\sM_{0,n}^\Sigma(\P^2, d)^0$ as the open  substack of  $\bar{\sM}_{0,n}^\Sigma(\P^2, d)$ parametrizing pairs $(f:C\to \P^2, \mathfrak{d})$ such that $f_*(\mathfrak{d})$ is a reduced 0-cycle of degree $n$.  

 At a geometric point $x:=(f:C\to \P^2, \mathfrak{d})$ such that $f:C\to f(C)$ is birational, there are no automorphisms of $x$, so $\bar{\sM}_{0,n}^\Sigma(\P^2, d)$ is a scheme in a neighborhood of $x$. Similarly, $\bar{\sM}_{0,n}^\Sigma(\P^2, d)^0$ is a scheme in a neighborhood of each geometric point $x=(f:C\to \P^2, \mathfrak{d})$ such that $f$ is unramified along the support of $\mathfrak{d}$.

We consider additional open substacks of $\bar{\sM}_{0,n}^\Sigma(\P^2, d)^0$. Let 
$\sM_{0,n}^\Sigma(\P^2, d)^0=\sM_{0,n}^\Sigma(\P^2, d)\cap \bar{\sM}_{0,n}^\Sigma(\P^2, d)^0$. Let  $\bar{\sM}_{0,n}^\Sigma(\P^2, d)_\delta\subset \bar{\sM}_{0,n}^\Sigma(\P^2, d)$ be the open substack of   $(f, C,\mathfrak{d})$ such that\\[5pt]
i. the restriction of $f$ to each irreducible component of $C$ is unramified, \\[3pt]
ii.  $f$ is an embedding in a neighborhood of each singular point of $C$\\[3pt]
iii. the geometric image $f(C)$ has only ordinary double points as singularities.\\[5pt]
We set 
\begin{align*}
\bar{\sM}_{0,n}^\Sigma(\P^2, d)_\delta^0&:=\bar{\sM}_{0,n}^\Sigma(\P^2, d)_\delta\cap \bar{\sM}_{0,n}^\Sigma(\P^2, d)^0\\
\sM_{0,n}^\Sigma(\P^2, d)_\delta&:=\bar{\sM}_{0,n}^\Sigma(\P^2, d)_\delta\cap \sM_{0,n}^\Sigma(\P^2, d)\\
\sM_{0,n}^\Sigma(\P^2, d)^0_\delta&:=
\bar{\sM}_{0,n}^\Sigma(\P^2, d)^0_\delta\cap \sM_{0,n}^\Sigma(\P^2, d).
\end{align*}

For $i\ge 2$, we have the closed substacks $D^{(i)}\subset \bar{\sM}_{0,n}^\Sigma(\P^2, d)$ with geometric points consisting of those curves with $\ge i$ irreducible components, and the open substack $D^{(i)}_\delta:=D^{(i)}\cap \bar{\sM}_{0,n}^\Sigma(\P^2, d)_\delta$ of $D^{(i)}$.

We let $\bar{\sM}_{0,n}^\Sigma(\P^2, d)_{cusp}$, $\bar{\sM}_{0,n}^\Sigma(\P^2, d)_{tac}$, $\bar{\sM}_{0,n}^\Sigma(\P^2, d)_{tri}$ be the respective closures in $\bar{\sM}_{0,n}^\Sigma(\P^2, d)$ of the locally closed substacks of $\sM_{0,n}^\Sigma(\P^2, d)$ with geometric points those $f:C\to \P^2$ such that $f(C)$ has degree $d$ and  $f(C)$ has a single ordinary cusp, resp. a single ordinary tacnode, resp. a single ordinary triple point, and all other singular points ordinary double points. Here , a single ordinary cusp, resp., ordinary tacnode, resp. ordinary triple point are singularities of the respective forms
\[
y^2-x^3,\ y^2-x^4, \ xy(x-y)
\]
in $\Spec K[[x,y]]$, with $C$ defined over the algebraically closed field $K$.

For $?\in \{cusp, tac, tri\}$, we set 
\begin{align*}
\bar{\sM}_{0,n}^\Sigma(\P^2, d)^0_{?}&:=\bar{\sM}_{0,n}^\Sigma(\P^2, d)_{?}\cap \bar{\sM}_{0,n}^\Sigma(\P^2, d)^0,\\
\sM_{0,n}^\Sigma(\P^2, d)^0_{?}&:=\bar{\sM}_{0,n}^\Sigma(\P^2, d)^0_{?}\cap 
\sM_{0,n}^\Sigma(\P^2, d)
\end{align*}
and for $?\in \{tac, tri\}$, we set 
\begin{align*}
\bar{\sM}_{0,n}^\Sigma(\P^2, d)^0_{?,\delta}&:=\bar{\sM}_{0,n}^\Sigma(\P^2, d)_{?}\cap \bar{\sM}_{0,n}^\Sigma(\P^2, d)^0_\delta\\
\sM_{0,n}^\Sigma(\P^2, d)^0_{?,\delta}&:=\bar{\sM}_{0,n}^\Sigma(\P^2, d)^0_{?,\delta}\cap \sM_{0,n}^\Sigma(\P^2, d).
\end{align*}
We let $\bar{\sM}_{0,n}^\Sigma(\P^2, d)_{ord}$ be the union
\begin{multline*}
\bar{\sM}_{0,n}^\Sigma(\P^2, d)_{ord}\\
=
\sM_{0,n}^\Sigma(\P^2, d)^0_\delta\cup \sM_{0,n}^\Sigma(\P^2, d)^0_{cusp}\cup \sM_{0,n}^\Sigma(\P^2, d)^0_{tac,\delta}\\\cup \sM_{0,n}^\Sigma(\P^2, d)^0_{tri,\delta}\cup (D^{(2)}_\delta\setminus D^{(3)})
\end{multline*}

\begin{lemma}\label{lem:smoothness} Let $x=(f:C\to \P^1, \mathfrak{d})$ be a point of  $\bar{\sM}_{0,n}^\Sigma(\P^2, d)^0_\delta$.   Suppose in addition that $n\le 3d-1$ and either\\[5pt]
i. $C$ is irreducible \\[3pt]
or\\[3pt]
ii. $C$ has two irreducible components, $C=C_1\cup C_2$, $\mathfrak{d}=\mathfrak{d}_1+\mathfrak{d}_2$ with $\mathfrak{d}_i$ of degree $n_i$,  supported on $C_i$, and $n_i\le 3d_i$, $i=1, 2$.
\\[5pt]
 Then the map $\mathfrak{z}:\bar{\sM}_{0,n}^\Sigma(\P^2, d)^0\to (\Sym^n\P^2)^0$ is smooth at $x$.
\end{lemma}

\begin{proof}  To check that $\mathfrak{z}$ is smooth at $x$, we may pullback by the \'etale cover $(\P^2)^n_0\to (\Sym^n\P^2)^0$. This replaces $\bar{\sM}_{0,n}^\Sigma(\P^2, d)$ with $\bar{\sM}_{0,n}(\P^2, d)$, the  moduli stack of $n$-pointed genus 0 curves, and replaces $\mathfrak{d}$ with a tuple of points $(p_1,\ldots, p_n)$ on $C$.

The assumption that $x$ is in  $\bar{\sM}_{0,n}^\Sigma(\P^2, d)^0_\delta$ implies that the space of first order deformations of $f$ are given by $H^0(C, \omega_f)$, where $\omega_f$ is the invertible sheaf  $f^*\omega_{\P^2/k}^{-1}\otimes \omega_{C/k}$. We claim that the sum of restriction maps $i_{p_j}$ defines a surjection
\[
i_{p_*}^*:H^0(C, \omega_f)\to \oplus_{j}f^*T_{\P^2, f(p_j)}/T_{C, p_j}
\]
Indeed, we have the exact sheaf sequence
\[
0\to \omega_f(-\sum_{i,j}p_j^i)\to \omega_f\to  \oplus_{j}f^*T_{\P^2, f(p_j)}/T_{C_i, p_j}\to 0
\]
so the surjectivity we are seeking will follow from the vanishing of $H^1(C, \omega_f(-\sum_{j}p_j))$.  By Riemann-Roch, we have
\[
H^1(C, \omega_f(-\sum_{j}p_j))\cong H^0(C, f^*\sO_{\P^2}(-3)(\sum_{j} p_j))
\]

Suppose that $C$ is irreducible. Then $f^*\sO_{\P^2}(-3)(\sum_{j} p_j)$ has degree $n-3d<0$, hence has no global sections, $H^1(C, \omega_f(-\sum_{j}p_j))=0$ and the restriction map is surjective.

In case (ii),  $C=C_1\cup C_2$,  then we can order the $C_i$ so that $n_1\le 3d_1-1$, $n_2\le 3d_2$. Let $t$ be a global section of $f^*\sO_{\P^2}(-3)(\sum_{j} p_j)$. Since this invertible sheaf has negative degree $n_1-3d_1$ on $C_1$, $t\equiv0$ on $C_1$, and in particular, $t$ vanishes at the intersection point $p_0$ of $C_1$ and $C_2$. Thus the restriction of $t$ to $C_2$ gives a section of  $f^*\sO_{\P^2}(-3)(-p_0+\sum_{j} p_j)\otimes\sO_{C_2}$, which is an invertible sheaf of degree $n_2-3d_2-1<0$ on $C_2$ and thus $t$ is identically zero.

The deformation space of $f$, plus the direct sum of tangent spaces $\oplus_jT_{C,p_j}$ is equal to the tangent space to $\bar{\sM}_{0,n}^\Sigma(\P^2, d)$ at $x$. Since $f$ is unramified at each point $p_j$, 
\[
f^*T_{\P^2, f(p_j)}=T_{C,p_j}\oplus   f^*T_{\P^2, f(p_j)}/T_{C, p_j}, 
\]
and thus the surjectivity proven above yields the surjectivity of
\[
d\mathfrak{z}:T_{\bar{\sM}_{0,n}^\Sigma(\P^2, d),x}\to \mathfrak{z}^*T_{\Sym^n\P^2, \mathfrak{z}(x)},
\]
that is, $\mathfrak{z}$ is smooth at $x$.
\end{proof}

\begin{lemma}\label{lem:DimOrd} 1. $\bar{\sM}_{0,n}^\Sigma(\P^2, d)$ is irreducible of dimension $3d-1+n$.\\[5pt]
2. $\bar{\sM}_{0,n}^\Sigma(\P^2, d)_{ord}\subset
\bar{\sM}_{0,n}^\Sigma(\P^2, d)$ is an open substack of $\bar{\sM}_{0,n}^\Sigma(\P^2, d)$\\[5pt]
3. Each irreducible component  $F$ of $\bar{\sM}_{0,n}^\Sigma(\P^2, d)\setminus \bar{\sM}_{0,n}^\Sigma(\P^2, d)_{ord}$ has codimension $\ge 2$ in $\bar{\sM}_{0,n}^\Sigma(\P^2, d)$.\\[5pt]
4. For $n\le 3d-1$, the map $\mathfrak{z}:\bar{\sM}_{0,n}^\Sigma(\P^2, d)\to \Sym^n(\P^2)$ is proper, surjective and generically smooth, with generic fiber of dimension $3d-1-n$. In particular, for $n=3d-1$, and for each irreducible component $F$ as in (3), we have 
\[
\codim_{\Sym^{3d-1}(\P^2)}\mathfrak{z}(F)\ge 2
\]
\end{lemma}

\begin{proof} This is proven in a more general setting in \cite{J} and in \cite{KP}; in the case of $\P^2$, the proof is fairly straightforward: The fact that the tangent bundle of $\P^2$ is ample implies that each stable map $f:C\to \P^2$ of an $n$-pointed semi-stable genus zero curve $C$ deforms to a stable map of an $n$-pointed  irreducible curve $\tilde{C}\to \P^2$. For $C=\P^1$, a morphism $f:\P^1\to \P^2$ is given by three homogeneous polynomials of the same degree, $f=(f_0:f_1:f_2)$, which we may assume have no common zeros in $\P^1$. For $f([C])$ to have degree $d$, this means that the $f_i$ all have degree $d$; this gives the irreducibility. An open subset of such triples define morphisms $f$ with $C\to f(C)$ birational and among such the dimension of the moduli is given by $3(d+1)-4=3d-1$.

For (2), if $f:C\to \P^2$ is in $\bar{\sM}_{0,n}^\Sigma(\P^2, d)_{ord}$, it is easy to see that a small deformation of such maps lie in $\sM_{0,n}^\Sigma(\P^2, d)^0_\delta$, hence $\bar{\sM}_{0,n}^\Sigma(\P^2, d)_{ord}$ is open. 

For (3), using our dimension count for $\bar{\sM}_{0,n}^\Sigma(\P^2, d)$, with $d$ replaced by $d'<d$, it is easy to see that $D^{(i)}$ has codimension $\ge 2$ for all $i\ge 3$. Similarly, the deformations of maps with higher order singularities are all unobstructed, so one can calculate the codimension of maps $f$ with $f(C)$ having singularities other than ordinary nodes or at most one cusp, one tacnode or one triple point by computing the relevant local deformation space, and one again finds these all have codimension $\ge2$.

For (4), $\bar{\sM}_{0,n}^\Sigma(\P^2, d)$ is proper over $k$, hence $\mathfrak{z}:\bar{\sM}_{0,n}^\Sigma(\P^2, d)\to \Sym^n(\P^2)$ is proper.  Take $n\le 3d-1$,  let $f=(f_0: f_1: f_2):\P^1\to \P^2$ be  a morphism with the $f_i$ degree $d$ polynomials such that $f(\P^1)$ has degree $d$ and only ordinary nodes. Choose points $p_1,\ldots, p_n\in \P^1$ suitably general, so that $f(p_i)\neq f(p_j)$ for $i\neq j$ and with $f(p_i)$ a smooth point of $f(\P^1)$ for all $i$. The normal bundle of the map $f$ is then $\sO_{\P^1}(3d-2)$ and since $n\le 3d-1$,  the restriction map
\[
H^0(\P^1, N_f)\to \oplus_{i=1}^ni_{p_i}^*N_f
\]
is surjective.  This shows that for each first order deformation $q_{i,\epsilon}$  of the $q_i:=f(p_i)$ in $\P^2$, $i=1,\ldots, n$, there is a corresponding first order  deformation $(f_\epsilon, p_{1,\epsilon},\ldots, p_{n,\epsilon})$ of $(f, p_1,\ldots, p_n)$ with $f_\epsilon(p_{i,\epsilon})=q_{i,\epsilon}$, and thus the map $\mathfrak{z}$ is a smooth morphism in the neighborhood of $(f, p_1,\ldots, p_n)$. Thus $\mathfrak{z}$ is dominant, hence surjective; as $\bar{\sM}_{0,n}^\Sigma(\P^2, d)$  has dimension $3d-1+n$ and 
$\Sym^n(\P^2)$ has dimension $2n$, this gives the generic fiber dimension of $3d-1-n$. In particular, for $n=3d-1$, $\mathfrak{z}$ is generically finite. The last assertion in (4) is an immediate consequence of this and (3).
\end{proof}

\begin{definition} Fix an integer $d\ge3$. Let $\Sym^{3d-1}(\P^2)^0_{dgn}\subset 
\Sym^{3d-1}(\P^2)^0$ be the closed subset
\[
\Sym^{3d-1}(\P^2)^0_{dgn}:=\mathfrak{z}(\bar{\sM}_{0,n}^\Sigma(\P^2, d)^0\setminus \sM_{0,n}^\Sigma(\P^2, d)_\delta^0);
\]
as $\mathfrak{z}$ is proper, $\Sym^{3d-1}(\P^2)_{dgn}$ is in fact closed. We let
$\Sym^{3d-1}(\P^2)^0_\delta$ be the open subscheme $\Sym^{3d-1}(\P^2)^0\setminus \Sym^{3d-1}(\P^2)^0_{dgn}$  of $\Sym^{3d-1}(\P^2)^0$.
\end{definition}

\begin{remark}
In words, a geometric point $\mathfrak{d}$ of $\Sym^{3d-1}(\P^2)^0$ is in $\Sym^{3d-1}(\P^2)^0_\delta$ exactly when each $x:=(f_x:C_x\to \P^2_{k(x)}, \mathfrak{d})$ in $\mathfrak{z}^{-1}(\mathfrak{d})$ satisfies\\[5pt]
i.  $C_x$ smooth,  \\[3pt]
ii. $f_x:C_x\to f(C_x)$ birational \\[3pt]
iii. $f_x(C_x)_{\overline{k(x)}}$  has only ordinary double points as singularities. 
\end{remark}

\begin{lemma}\label{lem:DeltaProps} 1. The map 
\[
\mathfrak{z}: \mathfrak{z}^{-1}(\Sym^{3d-1}(\P^2)^0_\delta)\to (\Sym^{3d-1}(\P^2)^0_\delta
\]
is \'etale.\\[5pt]
2.  The generic points of $\Sym^{3d-1}(\P^2)_{degn}$ that have codimension one on 
$\Sym^{3d-1}(\P^2)^0$ are the points $\mathfrak{z}(\eta_{?})$ for 
\[
?\in \{cusp, tac, tri,\}\cup  \{(d_1, n_1)\mid 1\le d_1\le d/2, n_1=3d_1, 3d_1-1\}.
\]
\end{lemma}

\begin{proof} (1) follows from Lemma~\ref{lem:smoothness} and (2) follows from Lemma~\ref{lem:DimOrd}, since for $n=3d-1$, the map $\mathfrak{z}$ has relative dimension zero.
\end{proof}

\section{The algebraic Welschinger invariant} Fix a field $K$, not necessarily algebraically closed, and let $C\subset \P^2_K$ be a geometrically integral, and geometrically rational curve of degree $d$, such that $C_{\bar{K}}$ has only ordinary double points as singularities.  Let $p_1,\ldots, p_r$ be the singular points of $C$. Fix a defining equation $F\in K[X_0, X_1, X_2]_d$ for $C$. 

Fix $p=p_\ell$. Choose a linear form $L\in k[X_0, X_1, X_2]$ with $L(p)\neq0$ and choose local coordinates $t_1=t_1^\ell, t_2=t_2^\ell\in \sO_{\P^2, p}$. Let $f=F/L^d\in \mathfrak{m}_p\subset
\sO_{\P^2, p}$. Then we have the Hessian matrix
\[
\left(\frac{\del^2 f}{\del t_i\del t_j}\right).
\]
The fact that $p$ is an ordinary double point implies $\det\left(\del^2 f/\del t_i\del t_j\right)(p)\neq0$, which gives us the local invariant
\[
e_p(C):=\left<\det\left(\frac{\del^2 f}{\del t_i\del t_j}\right)(p)\right>\in \GW(K(p)).
\]
If we change $L$ to $L'$, then the matrix $\left(\del^2 f/\del t_i\del t_j\right)(p)$ changes to  $(L/L')^d(p)\cdot\left(\del^2 f/\del t_i\del t_j\right)(p)$, so the determinant changes by a square and the resulting element of $\GW(K(p))$  is unchanged. If we change coordinates, say to $s_1, s_2$, then the Hessian matrix changes by
\[
\left(\frac{\del^2 f}{\del s_i\del s_j}\right)(p)=\left(\frac{\del t_i}{\del s_j}\right)^t(p)\left(\frac{\del^2 f}{\del t_i\del t_j}\right)(p)\left(\frac{\del t_i}{\del s_j}\right)(p),
\]
so again the Hessian determinant changes by a square. Finally, if we replace $F$ with $\lambda\cdot F$, $\lambda\in K^\times$, then  the Hessian determinant changes by $\lambda^2$. Thus,  the local invariant $e_p(C)$ is  well-defined, independent of the choices we have made.

We define $\Wel_K(C)$ to be  the 1-dimensional form
\[
\Wel_K(C) = \left<\prod_\ell \Nm_{K(p_\ell)/K} \det\left(\frac{\del^2 f}{\del t_i^\ell\del t^\ell_j}\right)(p_\ell)\right>.
\]
As the individual determinants  $\det\left(\frac{\del^2 f}{\del t_i^\ell\del t^\ell_j}\right)(p_\ell)$  are well-defined modulo squares in $K(p_\ell)^\times$, the product of norms is well-defined modulo squares in $K^\times$, so the one-dimensional quadratic form  $\Wel_K(C) \in \GW(K)$ is well-defined.
It is clear from the construction that $\Wel_K(C)$  is functorial with respect to field extensions.

\begin{lemma}\label{WelschFunct} Let $C\subset \P^2_K$ be a geometrically reduced plane curve such that $C_{\bar{K}}$ is rational and has only ordinary double points as singularities, let $L\supset K$ be an extension field and let $X_L:=C\times_KL$. Then $\Wel_L(C_L)\in \GW(L)$ is the image of $\Wel_K(C)\in \GW(K)$ under the base-extension map $\GW(K)\to \GW(L)$.
\end{lemma}

An essential aspect of the algebraic Welschinger invariant is the following extension property.

\begin{lemma}\label{lem:WelschExt} Let $\sO$ be a DVR containing $k$, with quotient field $K$ and residue field $\kappa$ and let $\pi:\sC\subset \P^2_{\sO}$ be a flat family of plane curves over $\Spec \sO$ defined by a homogeneous $F\in \sO[X_0, X_1, X_2]$.  Let $\sC_\sing\subset \sC$ be the subscheme of singularities of $\sC$, that is, the closed subscheme defined by the ideal $(F, \del F/\del X_0,\del F/\del X_1, \del F/\del X_2)$. Suppose that $\sC_\sing$ contains a closed subscheme $\sD\subset \sC_\sing$ such that\\[5pt]
i. $\sD$ is \'etale over $\sO$ \\[3pt]
ii. $\sD_K=\sC_{\sing, K}$\\[3pt]
iii. Each $x\in \sD_k\subset \sC_k$ is an ordinary double point of $\sC_k$ and each point $y\in \sD_K$ is an ordinary double point of $\sC_K$\\[5pt]
 Then $\Wel_K(\sC_K)\in \GW(K)$ extends to an element $\Wel(\sD)\in \GW(\sO)$. If in addition $\sC_\sing$ satisfies (i)- (iii), then taking $\sD:=\sC_\sing$,  $\Wel(\sD)$ has image $\Wel_\kappa(\sC_\kappa)$ in $\GW(\kappa)$.
\end{lemma}

\begin{proof} From the purity theorem for $\GW$ \cite[\S 2]{MorelA1Top}, we have the short exact sequence
\[
0\to \GW(\sO)\to \GW(K)\xr{\del_t}\GW(\kappa)\to 0
\]
where $t\in \sO$ is a uniformizing parameter. Letting $\sO^h$ denote the henselization of $\sO$ and $K^h$ the quotient field of $\sO^h$, we have the commutative diagram
\[
\xymatrix{
\GW(K)\ar[r]^{\del_t}\ar[d]&\GW(\kappa)\ar@{=}[d]\\
\GW(K^h)\ar[r]_{\del_{t^h}}&\GW(\kappa)
}
\]
where $t^h\in \sO^h$ is the image of $t$. Thus, we may suppose that $\sO$ is henselian. 

In that case, the closed points $x_1,\ldots, x_m$ of $\sD_k$ lift to subschemes $\tilde{x}_1,\ldots, \tilde{x}_m$ of $\sC_\sing$ such that $\sD=\amalg_{i=1}^m\tilde{x}_i$,  each $\tilde{x}_i$ is \'etale over $\sO$ and $x_i=\tilde{x}_i\times_\sO \kappa$. In addition, from (iii), we know that $\sC_K$ has only ordinary double points as singularities and from (ii),  that $\sD_K=\sC_{K, \sing}$.

We may choose a linear form $L\in \sO[X_0, X_1, X_2]$ with $L=0$ disjoint from  $\sC_\sing$ and local parameters $t^\ell_1, t^\ell_2$ for $\sO_{\P^2_\sO, \tilde{x}_\ell}$ for each $\ell$,  and use these to compute the Hessian determinant $\det(\del^2f/\del t_i^\ell \del t_j^ \ell)$ of $f:=F/L^d$ at each $\tilde{x}_\ell$. Since $(\tilde{x}_\ell)\times_{\sO}\kappa$ is an ordinary double point of $\sC_\kappa$, it follows that $\det(\del^2f/\del t_i^\ell \del t_j^ \ell)(\tilde{x}_\ell)$ is a unit in $\sO_{\tilde{x}_\ell}$ for each $\ell$ and thus the one-dimensional form
\[
\left <\prod_{\ell=1}^m\Nm_{\sO_{\tilde{x}_\ell}/\sO}\det\left(\frac{\del^2f}{\del t_i^\ell \del t_j^ \ell}\right)(\tilde{x}_\ell)\right>\in \GW(\sO)
\]
gives an element  $\Wel_\sO(\sD)\in \GW(\sO)$. By conditions (ii), (iii),   $\Wel_K(\sC_K)\in\GW(K)$ is the image of $\Wel_\sO(\sD)$ under the natural map $\GW(\sO)\to \GW(K)$. If  $\sC_\sing$ satisfies (i)-(iii), then $\Wel_\sO(\sC_\sing)\in \GW(\sO)$ maps to $\Wel_\kappa(\sC_\kappa)$ in $\GW(\kappa)$, by the compatibility of $\Nm$ and restriction for finite \'etale maps of semi-local schemes. 
\end{proof}

Let $K$ be a field, $S$ a reduced  $K$-scheme essentially of finite type over $K$ and let $\sC\subset S\times\P^2$ be a flat family of plane curves over $S$. Let $U_i\subset \P^2$ be the affine open subscheme $X_i\neq0$. Let $\sC_\sing\subset \sC$ be the ``relative singular locus'' of $\sC$, that is, if $V\subset S$ is an open subscheme and $F_V\in \sO_S(V)[X_0, X_1, X_2]$ is a defining equation for $V\times_S\sC$ in $V\times\P^2$, then   $V\times_S\sC_\sing$ is the closed subscheme of $\sC\cap V\times U_i$ defined by on $(F_V/\del X_j,  F_V/\del X_k)$, where $\{i,j,k\}=\{0,1,2\}$.

\begin{lemma}\label{lem:EtCrit} Suppose that $\Char K\neq 2$ and that $S$ is smooth over $K$
Let $\sD\subset \sC$ be a closed subscheme which is an open and closed component of $\sC_\sing$. Suppose that for each geometric point $\bar{s}$ of $S$, each point  $p\in \sD_{\bar{s}}$ is an ordinary double point of $\sC_{\bar{s}}$. Suppose in addition that $\sD$ is locally equi-dimensional over $S$ of relative dimension zero. Then  $\sD\to S$ is \'etale.

In particular, to check the condition (1) of Lemma~\ref{lem:WelschExt}, we need only see that $\sD$ is equi-dimensional over  $\Spec\sO$.
\end{lemma}

\begin{proof} Let $\bar{s}$ be a geometric point of $S$ and take $p\in \sC_{\bar{s}}$ an ordinary double point. Suppose $p\in V\times U_0$ for some open $V\subset S$ such that $\sC\cap V\times \P^2$ admits a defining equation $F_V\in \sO_X(V)[X_0, X_1, X_2]$. Let $\sJ\subset V\times U_0$   be the closed subscheme  defined by $(\del F_V/\del X_1, \del F_V/\del X_2)$. Then $\sC_\sing\cap V\times U_0=(\sC\cap V\times U_0)\cap \sJ$.   

In addition, the fact that $p\in \sC_{\bar{s}}$ is an ordinary double point  and the characteristic is different from 2 implies that 
$\sJ_{\bar{s}}=p$ in a neighborhood of $p$, and thus $\sJ$ has codimension two in $V\times U_0$ in   a neighborhood $W$ of $p$  in $V\times U_0$, that is $\sJ$ is a local complete intersection of codimension two in $W$. 

The fact that  $\sJ_{\bar{s}}=p$ in a neighborhood of $p$ in $\P^2_{k(\bar{s})}$ and as $\sJ$ is a local complete intersection of codimension two in $W$ implies that $\sJ\cap W$ is flat over $S$ and that  the canonical map
\[
\mathfrak{m}_{S,\bar{s}}/\mathfrak{m}_{S,\bar{s}}^2\to  \mathfrak{m}_{\sJ,p}/\mathfrak{m}_{\sJ,p}^2
\]
is an isomorphism, that is,  shrinking $W$ if necessary,  $\sJ\cap W$ is \'etale over $S$. In particular, $\sJ\cap W$ is reduced. If now $p$ is a point of $\sD_{\bar{s}}$ then after shrinking $W$ we have the 
\[
 \sD\cap W\subset \sC_\sing\cap W=\sC\cap \sJ\cap W\subset \sJ\cap W
\]
Since $\sD$ is a closed subscheme of $\sC$, $\sD$ is proper over $S$. Under the assumption that $\sD\to S$ is locally equi-dimensional of relative dimension zero, the fact that $\sJ\cap W$ is \'etale over $S$ implies that the closed subscheme $\sD\cap W$ of $\sJ\cap W$ contains an open neighborhood of $p$ in $\sJ\cap W$, and thus, after shrinking $W$ again, we have $\sD\cap W=\sJ\cap W$, and thus $\sD\cap W$ is \'etale over $S$.   
\end{proof}

\begin{remark} Suppose we are in the situation of Lemma~\ref{lem:WelschExt}, but only (i) and (iii) hold. We can form a ``partial invariant'' $\Wel_K(\sC_K,\sD_K)\in \GW(K)$ by 
\[
\Wel_K(\sC_K,\sD_K):=\prod_{\ell,\  p_\ell\in \sD_K} \Nm_{K(p_\ell)/K}e_{p_\ell}(\sC_K).
\]
Then the proof of Lemma~\ref{lem:WelschExt} shows that $\Wel_K(\sC_K,\sD_K)$ extends to an element $\Wel(\sC, \sD)\in \GW(\sO)$ and $\Wel(\sC, \sD)$ has image $\Wel(\sC_\kappa, \sD_\kappa)$ in $\GW(\kappa)$.
\end{remark}

\begin{remark}\label{rem:WelComp1} Take a curve $C\subset \P^2_\R$ and assume that $C$ is geometrically rational and integral and has only ordinary double points. If $p\in C$ is a double point with $\R(p)=\C$, then the local invariant $\Nm_{\C/\R}(e_p(C))$ is a positive real number, so $\<\Nm_{\C/\R}(e_p(C))\>=\<+1\>\in \GW(\R)$. If $\R(p)=\R$, there are two cases: the local defining equation in suitable analytic coordinates is either  $x^2-y^2=0$ or $x^2+y^2=0$. In the first case, $e_p(C)=\<-4\>=\<-1\>$ and in the second $e_p(C)=\<4\>=\<+1\>$. Thus $\Wel_\R(C)=\<(-1)^r\>\in \GW(\R)$ where $r$ is the number of real non-isolated double points and has signature $(-1)^r$.  This signature does not agree with Welschinger's mass, which is $(-1)^s$, $s$ being the number of real  isolated double points. However, for fixed degree $d$, the arithmetic genus of $C$ is $p_a(C)=p_a(d):=(d-1)(d-2)/2$, and thus the rational curve $C$ has $p_a(d)$ double points over $\C$. Since the number of complex double points is even, we have $s+r\equiv p_a(d)\mod 2$, so the signature of our invariant $\Wel_\R(C)$  is simply $(-1)^{p_a(d)}$ times Welschinger's mass.
\end{remark}

\section{Maps and curves} 

We have the projective space $\P(H^0(\P^2, \sO_{\P^2}(d)))=\P^{N_d}$, $N_d=d(d+3)/2$, parametrizing plane curves of degree $d$, with universal curve 
\[
\xymatrix{
&\sC_d\ar[dl]\ar[dr]\\
\P^{N_d}&&\P^2
}
\]
Let $P_{0,d}\subset \P^{N_d}$ be the locally closed subscheme of curves $C$ with (geometric) normalization isomorphic to $\P^1$ and let $P_{\delta, 0,d}\subset P_{0,d}$ be the open subscheme of $P_{0,d}$ of curves $C$ which have only ordinary double points (after base-extension to an algebraically closed field). Let 
$\bar{P}_{0,d}\subset \P^{N_d}$ be the closure of $P_{0,d}$. We let $\bar\sC_{0,d}\subset \sC_d$ be the pull-back $\sC_d\times_{\P^{N_d}}\bar{P}_{0,d}$ and $\sC_{\delta, 0,d}\subset \sC_{0,d}\subset \bar{\sC}_{0,d}$ the open subschemes  $\sC_d\times_{\P^{N_d}}\times P_{\delta, 0,d}\subset \sC_d\times_{\P^{N_d}}P_{0,d}$. 

We map $\bar{\sM}_{0,n}^\Sigma(\P^2, d)$ to $\P^{N_d}$ by associating to $f:C\to \P^2$ the Cartier divisor $f_*([C])$ on $\P^2$. To be precise, the construction of $\bar{\sM}_{0,n}(\P^2,d)$ by \cite{AbramOort} gives a presentation of  $\bar{\sM}_{0,n}(\P^2, d)$ as a quotient stack,
\[
\bar{\sM}_{0,n}(\P^2, d)=[\PGL_{M+1}\backslash V]
\]
where $V$ is a smooth quasi-projective $k$-scheme with a universal $n$-pointed quasi-stable genus 0 curve $\pi:\sP_V\to V$, $s_i:V\to \sP_V$, $i=1,\ldots, n$ and stable map
$f_V:\sP_V\to \P^2_V$ over $V$ with $f_V^*\sO_{\P^2}(1)$ of degree $d$ over $V$. This entire data has a $\PGL_{M+1}$-action arising from an embedding $\sP_V\to \P^M_V$. 

We have the free sheaf $W_{V,d}:=H^0(\P^2, \sO_{\P^2}(d))\otimes_k \sO_V$ on $V$ and the map $f_V$  gives rise to the map 
\[
\res_{\sP_V}:W_{V,d}\to \pi_*f_V^*\sO_{\P^2}(d). 
\]
At each point $v\in V$, the kernel of $\res_{\sP_V}\otimes k(v)$ is the one-dimensional $k$-vector space spanned by a defining equation for the degree $d$ divisor $f_{V*}(\sP_{V}\otimes k(v))$ on $\P^2_{k(v)}$ and thus the kernel of $\res_{\sP_V}$ is an invertible sheaf $\sL_V$ on $V$. The $\PGL_{M+1}$-action on $\sP_V$ lifts to a $\GL_{M+1}$-action on $f_V^*\sO_{\P^2}(d)$, giving rise to a $\GL_{M+1}$-action on $\sL_V$ over the $\PGL_{M+1}$-action on $V$. Letting
\[
p_1: \bar{\sM}_{0,n}(\P^2, d)\times \P^2\to \bar{\sM}_{0,n}(\P^2, d),\ 
p_2: \bar{\sM}_{0,n}(\P^2, d)\times \P^2\to \P^2
\]
be the projections, $\sL_V$ defines by descent an invertible subsheaf $\sL$ of $p_{1*}(p_2^*\sO_{\P^2}(d))$, giving the relative Cartier divisor $\sC\subset  \bar{\sM}_{0,n}(\P^2, d)\times \P^2$  receiving the universal stable map $f:\sP\to  \bar{\sM}_{0,n}(\P^2, d)\times \P^2$ over $\bar{\sM}_{0,n}(\P^2, d)$.

Permuting the pointings $\{s_i\}$ gives a $\mathfrak{S}_n$-action on $\sP_V\to V$, commuting with the $\PGL_{M+1}$-action and so this whole picture descends to give a closed substack $\sC^\Sigma\subset \bar{\sM}_{0,n}^\Sigma(\P^2, d)\times \P^2$ receiving the universal stable map $f:\sP^\Sigma\to  \bar{\sM}_{0,n}^\Sigma(\P^2, d)\times \P^2$ over $\bar{\sM}_{0,n}^\Sigma(\P^2, d)$.

The family $\sC^\Sigma_d\subset \bar{\sM}_{0,n}^\Sigma(\P^2, d)\times \P^2$ arises from the universal family $\sC_d\subset \P^{N_d}\times\P^2$ via a morphism 
$\bar{\sM}_{0,n}^\Sigma(\P^2, d)\to \P^{N_d}$. The image of the dense open substack  $\sM_{0,n}^\Sigma(\P^2, d)_\delta$ of $\bar{\sM}_{0,n}^\Sigma(\P^2, d)$ is contained in $P_{0,d}$ and as 
$\bar{\sM}_{0,n}^\Sigma(\P^2, d)$ is integral,   we have the morphism
\[
\mathfrak{D}:\bar{\sM}_{0,n}^\Sigma(\P^2, d)\to \bar{P}_{0,d}
\]
and the diagram
\begin{equation}\label{eqn:BasicDiagram}
\xymatrix{
&\sP^\Sigma_0\ar[rd]^{f^\Sigma_0}\ar[d]\\
&\bar{\sM}_{0,n}^\Sigma(\P^2, d)^0\ar[dr]^{\mathfrak{D}^0}\ar[dl]_{\mathfrak{z}}&\sC_{0,d}\ar[d]\ar@{^(->}[r]&\bar{P}_{0,d}\times\P^2\ar[dl]^{p_1}\\
\Sym^n(\P^2)^0&&\bar{P}_{0,d}
}
\end{equation}
Here $\sP^\Sigma_0$ is the restriction of the universal curve stack to $\bar{\sM}_{0,n}^\Sigma(\P^2, d)^0$, $\sC_{0,d}$ is the restriction of $\sC_d$ to  $\bar{P}_{0,d}$ and $f^\Sigma_0:\sP^\Sigma_0\to \sC_{0,d}$ is the restriction of the universal morphism.

Clearly $\mathfrak{D}$ is proper, and as $\mathfrak{D}$ has all nodal rational curves of degree $d$ in its image, $\mathfrak{D}$ is surjective. Thus $\mathfrak{D}^0$ is dominant.

\begin{definition} 1. Let $y$ be a point of $P_{\delta, 0,d}$. Then the corresponding curve $\sC_y\subset \P^2_{k(y)}$ has only ordinary double points and so its Welschinger invariant $\Wel_{k(y)}(\sC_y)\in \GW(k(y))$ is defined.\\[5pt]
2. Let $x$ be a point of $\sM_{0,n}^\Sigma(\P^2, d)_\delta^0$. Then $\mathfrak{D}(x)$ is in 
$P_{\delta, 0,d}$, we have the map of fields $\mathfrak{D}^*:k(\mathfrak{D}(x))\to k(x)$, and we set 
\[
\Wel_\delta(x):=\GW(\mathfrak{D}^*)(\Wel_{k(\mathfrak{D}(x))}(\sC_{\mathfrak{D}(x)}))\in \GW(k(x)).
\]
3. Take $n=3d-1$ and let $\mathfrak{d}$ be a point of $\Sym^{3d-1}(\P^2)^0_\delta$.  By  Lemma~\ref{lem:DeltaProps},  the map $\mathfrak{z}$ is \'etale over a neighborhood of  $\mathfrak{d}$,  so $\mathfrak{z}^{-1}(\mathfrak{d})$ is a finite set,
\[
\mathfrak{z}^{-1}(\mathfrak{d})=\{x_1,\ldots, x_r\},\ x_i\neq x_j \text{ for }i\neq j,
\]
the field extension $k(x_i)/k(\mathfrak{d})$ is separable  and $x_i$ is in  $\sM_{0,n}^\Sigma(\P^2, d)_\delta^0$ for each $i$. Define
\[
\Wel_d(\mathfrak{d}):=\sum_{i=1}^r\Tr_{k(x_i)/k(\mathfrak{d})}(\Wel_\delta(x_i))\in \GW(\mathfrak{d}).
\]
4. Let $\eta\in Sym^{3d-1}(\P^2)^0$ be the generic point. Define $\Wel_d\in \GW(k(\Sym^{3d-1}(\P^2)^0))$ by $\Wel_d:=\Wel_d(\eta)$.
\end{definition}

\begin{remark}\label{rem:WelComp2} It follows from Remark~\ref{rem:WelComp1} that for $\mathfrak{d}$ an $\R$-point of $\Sym^{3d-1}(\P^2)^0$, $\Wel_d(\mathfrak{d})\in \GW(\R)$ has signature $(-1)^{p_a(d)}$ times Welschinger's invariant for the signed count of real degree $d$ curves in $\P^2$ containing $\mathfrak{d}$. The rank of $\Wel_d(\mathfrak{d})$ is just the number of curves (over $\C$) containing $\mathfrak{d}$.
\end{remark}

Here is the main theorem of this section.

\begin{theorem}\label{thm:WelGlobal} 1. $\Wel_d\in \GW(k(\Sym^{3d-1}(\P^2)^0))$ extends to a global section
\[
\Wel_d\in H^0(\Sym^{3d-1}(\P^2)^0, \sGW)
\]
2. For $\mathfrak{d}\in \Sym^{3d-1}(\P^2)^0$, let 
\[
\ev_\mathfrak{d}:H^0(\Sym^{3d-1}(\P^2)^0, \sGW)\to \GW(k(\mathfrak{d}))
\]
be the evaluation map. Then 
\[
ev_\mathfrak{d}(\Wel_d)=\Wel_d(\mathfrak{d})
\]
for all  $\mathfrak{d}\in \Sym^{3d-1}(\P^2)^0_\delta$.
\end{theorem}

The proof of this result will require quite a bit a work, we begin the task with some elementary lemmas.

\begin{lemma}\label{lem:Etale1} The pullback of $\sC_{0,d}^\sing$ to $\sM_{0,n}^\Sigma(\P^2, d)^0_\delta$ is \'etale over $\sM_{0,n}^\Sigma(\P^2, d)^0_\delta$.
\end{lemma}

\begin{proof} If $D\subset \P^2$ is an irreducible reduced curve  of degree $d$ with only ordinary double points, defined over an algebraically closed field, such that the normalization is a $\P^1$, then $D$ has exactly $(d-1)(d-2)/2$ double points. If $D=f(\P^1)$, then looking at the local deformation theory of $\sM_{0,n}^\Sigma(\P^2, d)^0$ at a point $x:=(f, \mathfrak{d})$, we see that $\sM_{0,n}^\Sigma(\P^2, d)^0$ has no automorphisms at $x$ and is smooth at $x$, thus $\sM_{0,n}^\Sigma(\P^2, d)^0_\delta$ is smooth over $k$.  

The fact that the number of points  $\pi^{-1}(x)$ is constant in the neighborhood of $x$ implies that  $\sC_{0,d}^\sing$ is equi-dimensional over $\sM_{0,n}^\Sigma(\P^2, d)^0_\delta$ of relative dimension zero. By Lemma~\ref{lem:EtCrit}, 
$\sC_{0,d}^\sing$ is \'etale over $\sM_{0,n}^\Sigma(\P^2, d)^0_\delta$.
\end{proof}

The Grothendieck-Witt sheaf $\sGW$ on a smooth irreducible scheme $Y$ is an unramified sheaf, that is, for each $y\in Y$, the stalk $\sGW_y$ is the kernel of the collection of ``tame symbol maps'' at $y$ \cite[\S 2]{MorelA1Top}. More precisely, there is an exact sequence
\[
0\to \sGW_y\to \GW(k(Y))\xr{(\del_{t_x})_x}\oplus_{x\in Y^{(1)}_y}W(k(x))
\]
Here $ Y^{(1)}_y$ is the set of codimension one points $x\in Y$ with $y\in\overline{\{x\}}$, $t_x\in \sO_{Y,x}$ is a choice of a uniformizing parameter and $\del_{t_x}:|GW(k(Y))\to W(k(x))$ is the corresponding boundary, or ``tame symbol'', map. 

For a given codimension one point $y\in Y$, and an $\alpha\in \GW(k(Y))$, we say that $\alpha$ is {\em unramified at $y$} if $\del_t(\alpha)=0$, where $t\in \sO_{Y,y}$ is a parameter. This is independent of the choice of parameter, and is satisfied if and only if $\alpha$ extends to a (unique) element $\alpha_y\in \GW(\sO_{Y,y})$.

\begin{lemma}\label{lem:Etal2} The stack $\bar{\sM}_{0,n}^\Sigma(\P^2, d)_{ord}$ is a smooth $k$-scheme and the map 
\[
\mathfrak{z}:
\bar{\sM}_{0,3d-1}^\Sigma(\P^2, d)_{ord}\setminus  \bar{\sM}_{0,3d-1}^\Sigma(\P^2, d)^0_{cusp} \to 
\Sym^{3d-1}(\P^2)^0
\]
is   \'etale.
\end{lemma}

\begin{proof} For $x=(f, C, \mathfrak{d})\in\bar{\sM}_{0,3d-1}^\Sigma(\P^2, d)_{ord}$, the map $f:C\to f(C)$ is birational, hence has no automorphisms, and thus $\bar{\sM}_{0,3d-1}^\Sigma(\P^2, d)_{ord}$, is a $k$-scheme. The smoothness follows from the fact that the local deformation space is unobstructed and that there are no automorphisms. 

The proof that $\mathfrak{z}$ is \'etale at  $x:= (f, C, \mathfrak{d})\in\bar{\sM}_{0,3d-1}^\Sigma(\P^2, d)_{ord}\setminus \bar{\sM}_{0,3d-1}^\Sigma(\P^2, d)^0_{cusp}$ follows from the fact that over a suitable finite extension of $k(x)$, $C\cong \P^1$,   $f:\P^1\to \P^2$ is unramified and thus $N_f=\sO_{\P^1}(3d-2)$. Just as in the proof of Lemma~\ref{lem:DimOrd}, this shows that  $\mathfrak{z}$  is smooth at $x$, and hence \'etale by a dimension count.
\end{proof}

\begin{lemma}\label{lem:Ext1}
$\Wel_d\in \GW(k(\Sym^{3d-1}(\P^2)^0)$ extends to a global section
\[
\Wel_d\in H^0(\Sym^{3d-1}(\P^2)^0_\delta, \sGW)
\]
Moreover
\[
ev_\mathfrak{d}(\Wel_d)=\Wel_d(\mathfrak{d})
\]
for all  $\mathfrak{d}\in \Sym^{3d-1}(\P^2)^0_\delta$.
\end{lemma}

\begin{proof} We first show that $\Wel_d$ extends to a section as claimed. Let $Y=\Sym^{3d-1}(\P^2)^0_\delta$. By the description above of $\sGW$ as an unramified sheaf on the smooth $k$-scheme $Y$, we need only show that, for each codimension one point $x\in Y$,   $\Wel_d\in  \GW(k(Y))$ lifts to an element $\Wel_d^x\in \GW(Y_x)$, where $Y_x:=\Spec\sO_{\Sym^{3d-1}(\P^2)^0, x}$. 

By Lemma~\ref{lem:Etale1}, the map $\mathfrak{z}$ is \'etale on  $\bar{\sM}_{0,n}^\Sigma(\P^2, d)_{ord}$, so it suffices to show that
$\mathfrak{D}^{0*}\Wel(C_\eta)\in k(\bar{\sM}_{0,n}^\Sigma(\P^2, d))$ extends to section of $\sGW$ over $\tilde{Y}_x:=\bar{\sM}_{0,n}^\Sigma(\P^2, d)\times_YY_x$. Let $\sC_{\tilde{Y}_x}\to \tilde{Y}_x$ be the pull-back of the family of plane curves  $\sC_{0,d}\to \bar{\sM}_{0,n}^\Sigma(\P^2, d)^0$ to $\tilde{Y}_x$. By Lemma~\ref{lem:Etale1}, $\sC_{\tilde{Y}_x}\to \tilde{Y}_x$ satisfies the hypotheses of Lemma~\ref{lem:WelschExt}, which gives us the desired extension.

The statement on the evaluation map also follows from Lemma~\ref{lem:Etale1} and Lemma~\ref{lem:Etal2} by taking a chain of smooth closed subschemes of 
$\Spec \sO_{Sym^{3d-1}(\P^2)^0_\delta, \mathfrak{d}}$,
\[
\mathfrak{d}=D_0\subset D_1\subset\ldots\subset D_{r-1}\subset D_r=\Spec \sO_{Sym^{3d-1}(\P^2)^0_\delta, \mathfrak{d}},
\]
with $D_i\subset D_{i+1}$ of codimension one for each $i$. Letting $z_i$ be the generic points of $D_i$, applying  Lemma~\ref{lem:Etale1} and Lemma~\ref{lem:Etal2} shows that $\Wel_d(z_i)\in \GW(k(z_i))$ extends to an element $\Wel_d(\sO_{D_i, z_{i-1}})\in \GW(sO_{D_i, z_{i-1}})$ which maps to $\Wel_d(z_{i-1})\in \GW(k(z_{i-1}))$ under the canonical homomorphism $\GW(sO_{D_i, z_{i-1}})\to \GW(k(z_{i-1})$. By induction, $\Wel_d(z_i)=\ev_{z_i}(\Wel_d)$ and then the fact that $\sGW$ is an unramified sheaf implies that $\GW(\sO_{D_i, z_{i-1}})\subset \GW(k(z_i))$ and thus $\Wel_d(\sO_{D_i, z_{i-1}})$ is the restriction of $\Wel_d$ to $\Spec \sO_{D_i, z_{i-1}}$, whence follows $\ev_{k(z_{i-1})}(\Wel_d)=\Wel_d(z_{i-1})$.
\end{proof}

Let $\eta_{cusp}$ be   the generic point of 
$\bar{\sM}_{0,3d-1}^\Sigma(\P^2, d)_{cusp}$  and define $\eta_{tac}, \eta_{tri}$ similarly as the respective  generic points   of 
$\bar{\sM}_{0,3d-1}^\Sigma(\P^2, d)_{tac}$, $\bar{\sM}_{0,3d-1}^\Sigma(\P^2, d)_{tri}$.

There is one component of $D^{(2)}\setminus D^{(3)}$ for each way of writing $d=d_1+d_2$, $0\le d_1\le d_2\le d$ and each way of writing $3d-1=n_1+n_2$, $n_i\ge0$; we let $\eta_{d_1, n_1}$ be the corresponding generic point.

We postpone the main task of proving Theorem~\ref{thm:WelGlobal} to a series of subsequent sections; assuming these results, we give the proof here.

\begin{proof}[Proof of \hbox{Theorem~\ref{thm:WelGlobal}}] By Lemma~\ref{lem:Ext1}, $\Wel_d$ extends to a section 
\[
\Wel_d\in H^0(\Sym^{3d-1}(\P^2)^0_\delta, \sGW) 
\]
with $\ev_{\mathfrak{d}}\Wel_d=\Wel_d(\mathfrak{d})$ for all $\mathfrak{d}\in \Sym^{3d-1}(\P^2)^0_\delta$. Thus, we need only show that  $\Wel_d\in H^0(\Sym^{3d-1}(\P^2)^0_\delta, \sGW)$ extends to a section of $\sGW$ over all of  $\Sym^{3d-1}(\P^2)^0$. As $\sGW$ is definition an unramified sheaf, we need only check that $\Wel_d$ is unramified at  the codimension one points of $\Sym^{3d-1}(\P^2)^0$ that lie outside of  $\Sym^{3d-1}(\P^2)^0_\delta$.

 For  $?\in  \{(d_1, n_1)\mid 1\le d_1\le d/2, n_1=3d_1, 3d_1-1\}$ or $?\in \{cusp, tac, tri \}$, let $\bar{\eta}_{?}=\mathfrak{z}(\eta_{?})$. By Lemma~\ref{lem:DeltaProps}(2)  
\begin{multline*}
(\Sym^{3d-1}(\P^2)^0)^{(1)}\setminus \Sym^{3d-1}(\P^2)^0_\delta\\=\{ \bar{\eta}_{?}\mid
?\in  \{cusp, tac, tri \}\cup \{(d_1, n_1)\mid 1\le d_1\le d/2, n_1=3d_1, 3d_1-1\}\},
\end{multline*}
that is,  the codimension one points of $\Sym^{3d-1}(\P^2)^0$ that lie outside of 
$\Sym^{3d-1}(\P^2)^0_\delta$ are exactly  those points $\bar{\eta}_{?}$.

By Propositions~\ref{prop:UnramCusp}, \ref{prop:UnramTac}, \ref{prop:UnramTri}, and \ref{prop:UnramRed}, $\Wel_d$ is unramified at all points $\bar{\eta}_{?}$, which completes the proof.
\end{proof}

\begin{corollary}\label{cor:WelshInv} Let $K$ be an extension field of $k$ and let $\mathfrak{d}_1, \mathfrak{d}_2$ be in $\Sym^{3d-1}(\P^2)^0(K)$. If $\mathfrak{d}_1$ and $\mathfrak{d}_2$ lie in the same $\A^1$-connected component of $\Sym^{3d-1}(\P^2)^0_K$, then $\Wel_d(\mathfrak{d}_1)=\Wel_d(\mathfrak{d}_2)$. In particular, if 
$\mathfrak{d}_1, \mathfrak{d}_2$ are in $\Sym^{3d-1}(\P^2)^0_\delta(K)$ and lie in the same $\A^1$-connected component of $\Sym^{3d-1}(\P^2)^0_K$, then the corresponding ``quadratic count'' $\Wel_d(\mathfrak{d}_1)$ of degree $d$ rational curves with ordinary double points and containing $\supp(\mathfrak{d}_1)$ is the same as that for $\supp(\mathfrak{d}_2)$, $\Wel_d(\mathfrak{d}_2)$ .
\end{corollary}

\begin{proof} This follows from the $\A^1$-invariance of the Grothendieck-Witt sheaf (see \cite[Theorem 2.37]{MorelA1Top}): 
\[
H^0(X,\sGW)\cong H^0(X\times \A^1, \sGW), 
\]
which implies that, for $Y$ a smooth $k$-scheme and $\alpha\in H^0(Y_K, \sGW)$, then $\alpha$ is constant on $\A^1$-connected components of $Y(K)$. In detail, there is a Nisnevich sheaf $\pi_0^{\A^1}(Y)$ on $\Sm/k$ and a natural map
\[
cl_K:Y(K)\to \pi_0^{\A^1}(Y)_K
\]
where $\pi_0^{\A^1}(Y)_K$ denotes the stalk of $\pi_0^{\A^1}(Y)$ at the field $K$. Two $K$-points $x, y$ are said to lie in the same $\A^1$-connected component of $Y$ if $cl_K(x)=cl_K(y)$. 
\end{proof}

\begin{remark} Suppose we have two reduced divisors $\mathfrak{d}_1=\sum_{i=1}^s p_i$,  
$\mathfrak{d}_2=\sum_{j=1}^t q_j$, $p_i,  q_j$ closed points in $\P^2_K$, in the same $\A^1$-connected component of  $\Sym^{3d-1}(\P^2)^0(K)$ for some field $K$. In characteristic zero, the fact that $\A^1$ is simply connected implies that the ``field extension type'' of $\mathfrak{d}_1$ and $\mathfrak{d}_2$ agree, that is we have $s=t$ and after a reordering, we have $K(p_i)=K(q_i)$ for $i=1,\ldots, s$. In positive characteristic, $\A^1$ is not simply connected, and this invariance of the field invariance type will not in general hold.

We ask if, in characteristic zero, the $\A^1$-connected components of $\Sym^{3d-1}(\P^2)^0(K)$ are simply given by the field extension type. If so, this would give a direct generalization of Welschinger's theorem to arbitrary fields.
\end{remark}

\begin{ex}\label{ex:Deg123} We mention that for a smooth $k$-scheme and $K$ and extension field of $K$, a  {\em $\A^1$-chain connected component} of $Y(K)$ is a subset of an $\A^1$-connected component of $Y(K)$. Here $x, y\in  Y(K)$ lie in the same $\A^1$-chain connected component of $Y(K)$ if there are maps $f_i:\A^1_K\to Y_K$, $i=1,\ldots, n$ with $f_1(0)=x$, $f_n(1)=y$ and $f_i(1)=f_{i+1}(0)$ for $i=1,\ldots, n-1$. 

For example, it is easy to see that, if $\mathfrak{d}_1=\sum_{i=1}^rx_i+\mathfrak{d}_0$ and $\mathfrak{d}_2=\sum_{i=1}^ry_i+\mathfrak{d}_0$ are two reduced 0-cycles of degree $n$ on $\P^2_K$ with $x_i, y_j\in \P^2(K)$, then $\mathfrak{d}_1$ and $\mathfrak{d}_2$ lie in the same $\A^1$-chain connected component of $\Sym^n(\P^2)^0(K)$, assuming for instance that $K$ is an infinite field: just move each $x_i$ to a general $K$-point $z_i$ by  moving along a line $\ell_i\subset \P^2$, and then move $y_i$ to $z_i$ along a second line $\ell'_i$. Since $K$ is infinite, we may find a $z_i$ such that the resulting family lies in $\Sym^n(\P^2)^0$.

A similar argument works for points in a fixed degree two separable extension $L$ of $K$. Take  $x, y\in \P^2(L)$ with conjugates $x^\sigma$, $y^\sigma$, giving us closed points $p_x, p_y\in \P^2_K$ with $K(p_x)\cong K(p_y)\cong L$.  We may first move say $p_x$ into a general position with respect to $p_y$ and any other 0-cycle $\mathfrak{d_0}$, say by acting by a unipotent subgroup $U\cong \A^1$ in $\PGL_3$.  By taking $p_x$ suitably general with respect to $p_y$, we may assume that no three of the four $L$-points $x, y, x^\sigma, y^\sigma$ line on a line. 

Let $\ell_{x,y}$ be the line through $x, y$ with conjugate $\ell_{x,y}^\sigma=\ell_{x^\sigma, y^\sigma}$; by our general position assumption, we have  $\ell_{x^\sigma, y^\sigma}\neq \ell_{x, y}$. Then 
$\ell_{x,y}\cap \ell_{x,y}^\sigma$ is a $K$-point $z_0$ . Now consider the lines $\ell_{x,x^\sigma}$ and  $\ell_{y,y^\sigma}$ through $x, x^\sigma$, $y, y^\sigma$, which intersect in a $K$-point $z_1\neq z_0$. The projection from $z_1$ gives a map $\pi:\P^2_K\setminus\{z_1\}\to \P^1_K$, where we identify the target of the projection with $\P^1_K$ via $\pi(x)=\pi(x^\sigma)=0$, $\pi(y)=\pi(y^\sigma)=1$ and $\pi(z_0)=\infty$. This identifies $\ell_{x,y}\cup \ell_{x^\sigma, y^\sigma}\setminus \{z_0\}$ with $\A^1_L$ with $0$ mapping to $p_x$ and $1$ mapping to $p_y$, i.e., $p_x$ and $p_y$ are $\A^1$-chain connected. If we move $p_x$ into suitably general position as indicated above, the resulting $\A^1$-chain will exhibit $p_x+\mathfrak{d}_0$ and 
 $p_y+\mathfrak{d}_0$ as being in the same $\A^1$-connected component of $\Sym^n(\P^2)^0(K)$. 
 
 Thus, for $K$ of characteristic zero,  the $\A^1$ connected component of $\Sym^{3d-1}(\P^2)^0(K)$ of a 0-cycle $\mathfrak{d}_0$ with field extension type a collection of copies of $K$ or degree two extensions of $K$ is exactly the set of 0-cycles $\mathfrak{d}\in \Sym^{3d-1}(\P^2)^0(K)$ with the same field extension type as $\mathfrak{d}_0$. For $K$ of positive characteristic ($>2$) the $\A^1$ connected component containing $\mathfrak{d}_0$ contains at least all such $\mathfrak{d}$. 
  
We can push this one degree higher: suppose we have $p_x, p_y$ closed points of $\P^2_K$ with $K(p_x)=L=K(p_y)$ for $L$ a degree {\em three} extension of $K$. We want to show that, for with  $x+\mathfrak{d}_0, y+\mathfrak{d}_0\in
\Sym^{3d-1}(\P^2)^0(K)$, we have $\Wel_d(x+\mathfrak{d}_0)=\Wel_d(y+\mathfrak{d}_0)$. For this, $\GW(K)\to \GW(L)$ is injective since $L$ has odd degree over $K$, so it suffices to see that $x+\mathfrak{d}_0, y+\mathfrak{d}_0$ lie in the same $\A^1$-connected component of $\Sym^{3d-1}(\P^2)^0(L)$. We have $L\otimes_KL=L\times L'$ with $L'$ either a degree two extension of $K$ or $L'\cong L\times L$. But then $p_{xL}=x+p_x'$, $p_{yL}=y+p_y'$ with either $p_x'$ and $p_y'$ the sum of two $L$-points or $L(p_x')=L'=L(p_y')$ is a degree two extension of $L$, and the methods outlined above puts  $p_x'$ and $p_y'$ in the same $\A^1$-connected component of  $\Sym^{3d-1}(\P^2)^0(L)$.
\end{ex}

\section{Curves with a cusp}

We emphasize that we retain our assumption that $\Char k\neq 2,3$. 

Let $f:\P^1\to \P^2$ be a morphism with $f(\P^1)\subset \P^2$ a degree $d$ reduced curve and $f:\P^1\to f(\P^1)$ birational; we work over an algebraically closed base-field $k$. We suppose that $f(0)=(0:0:1)$ is an ordinary cusp on $f(\P^1)$ and that $f$ is unramified outside of $0$. We define the sheaf $\sN_f$ on $\P^1$ by exactness of the sequence
\begin{equation}\label{eqn:SheafSeq}
0\to T_{\P^1}\xr{df} f^*T_{\P^1}\to \sN_f\to 0
\end{equation}
Let $\sN_{f\text{tors}}\subset \sN_f$ be the torsion subsheaf and $\sN_f^0:=\sN_f/\sN_{f\text{tors}}$ the locally free quotient sheaf.  In suitable local analytic coordinates $t$ at $0\in \P^1$ and $(x,y)$ at $(0:0:1)$ in $\P^2$, we have $f(t)=(t^2, t^3)$ and so the map $df$ near 0 is given by the matrix $(2t, 3t^2)$.  We have the map $df:T_{\P^1}\to f^*T_{\P^2}$. Letting $\tilde{T}_{\P^1}\subset f^*T_{\P^2}$ be the invertible sheaf generated by the image of $df$, we have the short exact sheaf sequences
\[
0\to T_{\P^1}\xrightarrow{df}\tilde{T}_{\P^1}\to \sN_{f\text{tors}}\to 0
\]
and
\[
0\to \tilde{T}_{\P^1}\to f^*T_{\P^2}\to \sN_f^0\to 0
\]
We identify  $\sN_{f\text{tors}}$ with the ``limiting tangent line'' $L\subset \A^2_k$ defined by $y=0$; using the parameter $dx$ on $L$ gives the isomorphisms  $\sN_{f\text{tors}}\cong k(0)$ and  $\tilde{T}_{\P^1}\cong T_{\P^1}(0)\cong \sO_{\P^1}(3)$. As $\det f^*T_{\P^2}\cong \sO_{\P^1}(3d)$, we have $\sN_f^0\cong \sO_{\P^1}(3d-3)$.

\begin{proposition}\label{prop:CuspRam1} Let $C$ be a reduced irreducible genus zero curve, and let $f:C\to \P^2$ be a morphism, defined over a field $F$, with $f:C\to f(C)$ birational and $f(C)$ of degree $d$. Suppose that $f(C)$ has a cusp $f(p)$ for some $p\in C(F)$, and that $f$ is unramified on $C\setminus\{p\}$.  Then for every degree $3d-1$ reduced divisor $\mathfrak{d}$ on $C$, the map $\mathfrak{z}$ is ramified at $(C, f,\mathfrak{d})$. 

Moreover, let $(\Sym^n\P^2)^0_{cusp}\subset (\Sym^n\P^2)^0$ be the closure of image of  $\bar{\sM}_{0,n}^\Sigma(\P^2, d)_{cusp}$ und $\mathfrak{z}$. Then the restriction of $\mathfrak{z}$ to
\[
\mathfrak{z}_{cusp}:\bar{\sM}_{0,n}^\Sigma(\P^2, d)_{cusp}\to (\Sym^n\P^2)^0_{cusp}
\]
is unramified at $(C, f,\mathfrak{d})$.
\end{proposition}

\begin{proof} To compute the ramification, we may pull back by the \'etale cover $(\P^2)^{3d-1}_0\to \Sym^{3d-1}(\P^2)^0$ and consider the map $\tilde{\mathfrak{z}}:\sM_{0, 3d-1}(\P^2, d)\to (\P^2)^{3d-1}_0$ at $x:=(C, f, $ $(p_1\ldots, p_{3d-1}))$. Furthermore, we may extend the residue field $k(x)$ so that $C\cong \P^1$ and the cusp on $f(\P^1)$ is $f(0)=(0:0:1)$. Since the ramification locus is closed, we may assume that $p_i\neq0$ for all $i=1,\ldots, 3d-1$, so that $f$ is unramified at each $p_i$.

We have the exact sequence
\[
0\to \sN_f(-\sum_{i=1}^{3d-1}p_i)\to \sN_f\to \oplus_{i=1}^{3d-1}f^*T_{\P^2, f(p_i)}/T_{\P^1, p_i}\to 0
\]
By our computation of $\sN_f$ given above, we have
\[
\sN_f(-\sum_{i=1}^{3d-1}p_i)\cong i_{0*}k(0)\oplus \sO_{\P^1}(-2)
\]
and thus $H^1(\P^1_{k(x)}, \sN_f(-\sum_{i=1}^{3d-1}p_i))\cong k(x)$. Similarly $H^1(\P^1, \sN_f)=0$ and thus the map
\[
\res: H^0(\P^1, \sN_f)\to \oplus_{i=1}^{3d-1}f^*T_{\P^2, f(p_i)}/T_{\P^1, p_i}
\]
is not surjective. Identifying $f^*T_{\P^2, f(p_i)}$ with $f^*T_{\P^2, f(p_i)}/T_{\P^1, p_i}\oplus T_{\P^1, p_i}$,   we may identify $d\tilde{\mathfrak{z}}_{x}$ with the map
\[
\res\oplus \oplus_{i=1}^{3d-1}\id:H^0(\P^1, \sN_f)\oplus \oplus_{i=1}^{3d-1}T_{\P^1, p_i}\to \oplus_{i=1}^{3d-1}f^*T_{\P^2, f(p_i)}
\]
and we see thereby that $\tilde{\mathfrak{z}}$ is ramified at $x$.

To examine 
\[
\mathfrak{z}_{cusp}:\bar{\sM}_{0,n}^\Sigma(\P^2, d)_{cusp}\to (\Sym^n\P^2)^0_{cusp}
\]
(after pulling back to $(\P^2)^{3d-1}_0$ as above)
we note that 
\[
T_{(C,f,(p_1,\ldots, p_{3d-1})}=H^0(C, \sN_f^0)\oplus \oplus_{i=1}^{3d-1}T_{\P^1, p_i}
\]
 and as 
$\dim H^1(C, \sN_f^0(-\sum_i p))=1$, the map $H^0(C, \sN_f^0)\to 
\oplus_{i=1}^{3d-1}f^*T_{\P^2, f(p_i)}/T_{\P^1, p_i}$ is injective. This implies that 
$d\mathfrak{z}_{cusp}$ is injective at $(C,f,(p_1,\ldots, p_{3d-1})$.
\end{proof}

\begin{proposition}\label{prop:CuspRam2}  The map $\mathfrak{z}:\bar\sM_{0,3d-1}^\Sigma(\P^2, d)^0\to \Sym^{3d-1}(\P^2)^0$ has ramification order two along $\bar{\sM}_{0,3d-1}^\Sigma(\P^2, d)_{cusp}$.
\end{proposition}

\begin{proof}   As in the proof of Proposition~\ref{prop:CuspRam1}, we may replace $\mathfrak{z}:\sM_{0, 3d-1}^\Sigma(\P^2, d)^0\to \Sym^{3d-1}(\P^2)^0$ with $\tilde{\mathfrak{z}}:\sM_{0, 3d-1}(\P^2, d)^0\to (\P^2)^{3d-1}_0$. We consider a point $x:=(\P^1, f, (p_1\ldots, p_{3d-1}))$ of $\sM_{0, 3d-1}(\P^2, d)^0$ such that 
$f(0)$ is an ordinary cusp, $f:\P^1\setminus\{0\}\to \P^2$ is unramified, $f(\P^1)$ has only ordinary double points as singularities and such that $q_i:=f(p_i))$ is  supported in the smooth locus of $f(\P^1)$ for all $i$. We may extend the base field at will, so we may assume that the points $p_i$ are all $k$-points of $\P^1$.

For a morphism $\phi:Y\to X$ of smooth varieties over $k$ and a point $y\in Y$, we have the differential
\[
d\phi_y:T_{Y,y}\to \phi^*T_{X,x}
\]
Letting $\bar{\phi}^*T_{X,x}$ denote the cokernel of $d\phi_y$, we have the 2nd order differential
\[
d^2\phi:\Sym^2_{k(y)}T_{Y,y}\to \bar{\phi}^*T_{X,x}
\]
Indeed, the map $\phi^*:\sO_{X,x}\to \phi_*\sO_{y,y}$ induces the map on the jet spaces $\sJ^2\phi^*:\sJ^2\sO_{X,x}\to \phi_*\sJ^2\sO_{y,y}$ and thus $\sJ^2\phi^*$ induces map of the kernel of $d\phi:\Omega^1_{X,x}\to f_*\Omega^1_{Y,y}$ to the subspace $\phi_*\Sym^2\Omega^1_{Y,y}\subset  \phi_*\sJ^2\sO_{y,y}$; the map $d^2\phi$ is the dual of this map.

In our case, the sheaf sequence
\[
0\to \sN_f(-\sum_ip_i)\to \sN_f\to \oplus f^*T_{\P^2, q_i}/T_{\P^1, p_i}\to\0
\]
identifies $\bar{\mathfrak{z}}^*(T_{(\P^2)^{3d-1}, q_*})$ with $H^1(\P^1, \sN_f(-\sum_ip_i))$, and $\ker d\mathfrak{z}$ with the torsion subsheaf $i_{0*}k(0)\cong \sN_{f\text{tors}}$. We are therefore interested in computing the map
\[
d^2\mathfrak{z}_x:k(0)^{\otimes_{k(0)}2}\to H^1(\P^1, \sN_f(-\sum_ip_i))
\]
and showing that this map is non-zero. 

We may take $x$ to be any smooth point of $\bar{\sM}_{0,3d-1}^\Sigma(\P^2, d)_{cusp}$; to make the computation as simple as possible, we take $f$ of the form 
\[
f(t):=(1:f_1(t), f_2(t))=(t^2+\ldots, t^3+\ldots);\quad t:=t_1/t_0.
\]
and take $p_i=(1:t_i)$, $i=1,\ldots, 3d-1$ for $t_1,\ldots, t_{3d-1}$ general points in $k$.

We have already computed 
\[
H^1(\P^1, \sN_f(-\sum_ip_i))\cong H^1(\P^1, \sN_f(-\sum_ip_i)/\text{tors})\cong
H^1(\P^1, \sO(-2))\cong k.
\]
We detect the map $d^2\mathfrak{z}_x$ by composing with the Serre duality pairing
\[
H^1(\P^1, \sN_f(-\sum_ip_i))\times H^1(\P^1, K_{\P^1}\otimes \sN_f(-\sum_ip_i)^\vee)\to k.
\]

Take $a\in k^\times$. We define the 1-parameter deformation of $f$, $f_u:\P^1_{k[[u]]}\to \P^2_{k[[u]]}$ by
$f_u:=f+u\cdot g_a$, where
\[
g_a(t):=((a/t)f_1'(t), (a/t)f_2'(t))=(2a+\ldots, 3at+\ldots)
\]
Noting that $g_a(t)dt=(1/t)df(t)$  for all $t\neq0$, we see that the section of $\sN_f$ corresponding to the first-order deformation defined by $f_u$ goes to zero as a section of $\sN_f^0$. Similarly, $g_a(0)=(a, 0)$, so the section of $\sN_{f\text{tor}}\cong k(0)$ defined by $f_u$ is $2a$. Similarly, if we extend $f_u$ to a curve $x^1_u$ in $\bar\sM_{0,3d-1}^\Sigma(\P^2, d)^0$ by fixing the points $p_1,\ldots, p_{3d-1}$, we see that 
\[
d\mathfrak{z}(x^1_u)=(\ldots, (a/t(p_i))df(p_i)\ldots)
\]
Letting $p_{iu}\in \P^1_{k[[u]]}$ be the point  $p_{iu}:=(1:t_i-ua/t_i)$, we have 
$d\mathfrak{z}(x_u)=0$ and
\[
d^2\mathfrak{z}(x_u)=\del (\ldots, \frac{d^2}{du^2}(f_1(p_{iu}), f_2(p_{iu}))_{u=0}
,\ldots),
\]
where 
\[
\del:\oplus_{i=1}^{3d-1}f^*T_{\P^2, q_i}/T_{\P^1, p_i}\to H^1(\P^1, \sN_f(-\sum_{i=1}^{3d-1}p_i))
\]
is the boundary map in the cohomology sequence associated to \eqref{eqn:SheafSeq}. 

Letting $h_u(t,u)=f_u(t-(a/t)u)=(h_1(t, u), h_2(t,u))$, we have
\[
\frac{d^2}{du^2}(f_1(p_{iu}), f_2(p_{iu}))_{u=0}=\frac{d^2h_u}{du^2}(t_i, 0).
\]
Let $\A^2\subset \P^2$ be the affine open $X_0\neq0$ with standard coordinates $x=X_1/X_0$, $y=X_2/X_0$. Using the standard basis $(\del/\del x, \del/\del y)$ for $T_{\A^2}$,  we have the section $d^2h:=(d^2h_1/du^2(t,0)\cdot \del/\del x, 
d^2h_2/du^2(t,0)\cdot \del/\del y)$ of $f^*T_{\P^2}$ over $\P^1\setminus\{0,\infty\}$. Letting $\pi:f^T_{\P^2}\to \sN_f^0$ be the projection, one checks that $\pi(d^2h)$ extends to a section of $\sN_f^0$ over $\P^1\setminus\{0\}$.

We have the cover $\sU$ of $\P^1$ given by $U_0=\P^1\setminus\{0\}$, $U_1=\P^1\setminus\{p_1,\ldots, p_{3d-1}\}$. Since 
\[
 \frac{d^2}{du^2}h(t_i,0)= \frac{d^2}{du^2}(f_1(p_{iu}), f_2(p_{iu}))_{u=0}, 
 \]
 we can represent $d^2\mathfrak{z}(x_u)$ as the 1-cocycle in $C^1(\sU, \sN_f^0(\sum_i(-p_i)))$ given by $\pi(d^2h)\in H^0(U_0\cap U_1, \sN_f^0(\sum_i(-p_i))$. Thus for $\omega\in H^0(\P^1, \omega_{\P^1/k}\otimes(\sN_f^0)^\vee(\sum_ip_i))$, the pairing  $\<\omega, d^2\mathfrak{z}(x_u)\>$ is given by 
\[
 \sum_{i=1}^{3d-1}\Res_{p_i}\pi(d^2h)\cdot \omega
\]
where $\pi(d^2h)\cdot \omega$ is to be considered as a section of $\omega_{\P^1}$ over $U_0\cap U_1$ via the pairing 
\[
\omega_{\P^1/k}\otimes(\sN_f^0)^\vee(\sum_ip_i)\otimes 
\sN_f^0)(-\sum_ip_i)\to \omega_{\P^1/k}.
\]
We note that, as $\omega_{\P^1/k}\otimes(\sN_f^0)^\vee(\sum_ip_i)\cong \sO_{\P^1}$, the section $\omega_{\P^1/k}$ is non-zero at 0.

Since $\P^1\setminus U_0\cap U_1=\{0, p_1,\ldots, p_{3d-1}\}$, we can also compute $\<\omega, d^2\mathfrak{z}(x_u)\>$  as
\[
\<\omega, d^2\mathfrak{z}(x_u)\>=-\Res_0\pi(d^2h)\cdot \omega,
\]
since $\sum_{p\in \P^1}\Res_p\tau=0$ for each rational 1-form $\tau$ on $\P^1$. To make the computation of $\Res_0\pi(d^2h)\cdot \omega$, we use a trivialization of $\sN_f^0$ in a neighborhood of 0 given as follows: we have $\tilde{T}_{\P^1}=T_{\P^2}(0)$ and use as generator for $\tilde{T}_{\P^1}$ over $\P^1\setminus\{\infty\}$ the section $1/t\del/\del t$. Since 
\[
df(1/t\del/\del t)=((1/t)f_1'(t)\del/\del x, (1/t)f_2'(t)\del/\del y)=(2+\ldots, 3t+\ldots),
\]
sending a section $(a\del/\del x, b\del/\del y)$ of $f^*T_{\P^2}$ to $b\cdot(1/t)f_1'(t)-a\cdot (1/t)f_2'(t)$ descends to give an isomorphism of $\sN_f^0(-\sum_ip_i)$ with $\sO_{\P^1}$ over $\P^1\setminus\{p_1,\ldots, p_{3d-1}, \infty\}$. Via this isomorphism, $\omega$ will transform to a 1-form $v(t)dt$, with $v\in \sO_{\P^1,0}^\times$. 

A direct computation with respect to this local trivialization gives
\[
\pi(d^2h)\cdot\omega=(-\frac{6a^2}{t}+a_0+a_1t+\ldots)v(t)dt
\]
which yields
\[
-\Res_0\pi(d^2h)\cdot\omega=6v(0)\cdot a^2.
\]
Thus, the quadratic form $d^2\mathfrak{z}$ is, up to a non-zero scalar factor, the form $a\mapsto a^2$, and hence the ramification order is two.
\end{proof}

\begin{proposition} \label{prop:UnramCusp} Let  $\bar{\eta}_{cusp}=\mathfrak{z}(\eta_{cusp})$ . Then $\Wel_d\in H^0(\Sym^{3d-1}(\P^2)^0_\delta, \sGW)$ is unramified at $\bar{\eta}_{cusp}$.
\end{proposition}

\begin{proof} Let $L=k(\eta_{cusp})$ and let $K=k(\bar{\eta}_{cusp})$. Then $\eta_{cusp}$ is a smooth codimension one point on $\bar{\sM}_{0,3d-1}^\Sigma(\P^2, d)$ and $\bar{\eta}_{cusp}$ is a smooth codimension one point on $\Sym^{3d-1}(\P^2)$. Let $s$ be a uniformizing parameter at $\bar{\eta}_{cusp}$.

As $f_{\eta_{cusp}}$ has a single point of ramification, it follows that this point is an $L$-rational point of $C_{\eta_{cusp}}$, so we may assume $C_{\eta_{cusp}}=\P^1_L$, that the ramification point is $0=(1:0)$ and that $f_{\eta_{cusp}}(0)=(1:0:0)$. In suitable local analytic coordinates, we may write
\[
f_{\eta_{cusp}}(t)=(at^2, bt^3)
\]
for $a, b\in L^\times$. The local deformation theory of $f_{\eta_{cusp}}$ tells us that there is a local parameter $u$ at $\eta_{cusp}$ so that in an analytic neighborhood of $\eta_{cusp}$, the general map $f_u:\P^1\to \P^2$ is given by
\[
f_u(t)=(a(t^2-\frac{2}{3}u), b(t^3-ut)).
\]
For $u\neq0$, $f_u(\P^1)$ has the ordinary double point at $(au/3, 0)$, given by the fact that
\[
f_u(\sqrt{u})=(au/3, 0)=f_u(-\sqrt{u}).
\]

The local equation for $f_u(\P^1)$ at $(au/3, 0)$ is $0=y^2-ux^2+\ldots$ and so the contribution of the double point $(au/3, 0)$ to $\Wel(f_u(\P^1))$ is $u$. Since $f(\P^1)$ has the single cusp $(0,0)$ and all other singularities ordinary double points, the subscheme of the ``other'' singular points of $f_u$ forms a finite \'etale  cover of $\Spec L[[u]]$. It follows from a slight modification of Lemma~\ref{lem:WelschExt} that there is a unit $U\in L[[u]]^\times$ such that 
$\Wel(f_u(\P^1))=u\cdot U\in L((u))^\times$. Indeed, $U$ is just the norm of all the local invariants of the singularities of $f_u(\P^1)$ away from $(au/3, 0)$.

The map $\mathfrak{z}$ induces the finite flat map of rings
\[
\mathfrak{z}^*:K[[s]]\to L[[u]].
\]
and by a base-extension the finite flat ring homomorphism
\[
\mathfrak{z}_L^*:L[[s]]\to L[[u]].
\]
Since the map $\mathfrak{z}$ is ramified to order two along  $\eta_{cusp}$, $\mathfrak{z}_L^*(s)=w\cdot u^2$ for some unit $w\in L[[u]]^\times$ and thus $u$ satisfies a quadratic equation of the form
\[
u^2+P(s)su-Q(s)s=0
\]
for $P(s), Q(s)\in L[[s]]$ with $Q(0)\neq0$.  This gives us the element $\alpha:=u+s\cdot P(s)/2$ with $\alpha^2=v\cdot s$, $v$ a unit in $L[[s]]$.  Writing $U=U_1+V_1\cdot \alpha$, $U_1\in L[[s]]^\times$, $V_1\in L[[s]]$, we have
\begin{align*}
\Tr_{L((u))/L((s))}(u\cdot U)&=s(2V_1v -U_1P)\\
\Tr_{L((u))/L((s))}(u\cdot U\cdot\alpha)&=sv(2U_1-sV_1P)\\
\Tr_{L((u))/L((s))}(u\cdot U\cdot\alpha^2)&=s^2v(2V_1v-U_1P)
\end{align*}

We compute the   trace form $\Tr_{L((u))/L((s))}(\<u\cdot U\>)$ using the basis $1, \alpha$ for $L((u))$ over $L((s))$:
\begin{multline*}
\Tr_{L((u))/L((s))}(\<u\cdot U\>)=\begin{pmatrix}\Tr_{L((u))/L((s))}(u\cdot U)&
\Tr_{L((u))/L((s))}(\alpha\cdot u\cdot U)\\
\Tr_{L((u))/L((s))}(\alpha\cdot u\cdot U)&\Tr_{L((u))/L((s))}(\alpha^2\cdot u\cdot U)
\end{pmatrix}\\=
\begin{pmatrix} s(2V_1v -(U_1P))&sv(2U_1-sV_1P)\\
sv(2U_1-sV_1P)&s^2v(2V_1v-U_1P)
\end{pmatrix}
\end{multline*}
This is equivalent to 
\[
\begin{pmatrix} s(2V_1v -(U_1P))&v(2U_1-sV_1P)\\
v(2U_1-sV_1P)&v(2V_1v-U_1P)
\end{pmatrix},
\]
with determinant   $s(2V_1v -(U_1P))v(2V_1v-U_1P)-v^2(2U_1-sV_1P)^2$, a 
unit in $L[[s]]$. Thus $\Tr_{L((u))/L((s))}(\<u\cdot U\>)$ extends to a section of $\sGW$ over $\Spec L[[s]]$.

We now consider the entire pullback
\[
\bar{\sM}_{0,3d-1}^\Sigma(\P^2, d)^{\wedge_s}:=\bar{\sM}_{0,3d-1}^\Sigma(\P^2, d)\times_{\Sym^{3d-1}(\P^2)}\Spec K[[s]]\to \Spec K[[s]]
\]
Let $f^{-1}(\bar{\eta}_{cusp})_\red=\{x_1,\ldots, x_N\}$, where we take $x_1,\ldots, x_r$ to be points corresponding to maps $f$ such that $f(C)$ has a (single) cusp and $x_{r+1},\ldots, x_N$ the points corresponding to maps $f$ such that $f(C)$ has only ordinary double points. This breaks up $\bar{\sM}_{0,3d-1}^\Sigma(\P^2, d)^{\wedge_s}$ as a disjoint union
\[
\bar{\sM}_{0,3d-1}^\Sigma(\P^2, d)^{\wedge_s}=\amalg_{j=1}^N\Spec L_j[[u_j]]
\]

Let $\prod_{j=1}^NL_j[[s]]$ be the maximal unramified subextension in $K[[s]]\to \prod_jL_j[[u_j]]$.  For $j=r+1,\ldots, N$, the extension $K[[s]]\to L_j[[u_j]]$ is \'etale. In this  case, the scheme of double points $f_{u_j}(C_{L_j[[u_j]]})_\sing$ is \'etale over $\Spec L_j[[u_j]]$. It follows from  Lemma~\ref{lem:WelschExt}  that $\Wel_{L_j((u_j))}(f_{u_j}((C_{L_j((u_j))}))$ extends to a section of $\sGW$ over $\Spec L_j[[u_j]]$, and thus $\Tr_{L_j((u_j))/K((s))}\Wel_{L_j((u_j))}(f_{u_j}((C_{L_j((u_j))}))$ extends to a section of $\sGW$ over $\Spec K[[s]]$. For $j=1,\ldots, r$, we have just seen above that 
$\Tr_{L_j((u_j))/L_j((s))}\Wel_{L_j((u_j))}(f_{u_j}((C_{L_j((u_j))}))$ extends to a section of $\sGW$ over $\Spec L_j[[s]]$; since  the field extension $K\to L_j$ is separable by Proposition~\ref{prop:CuspRam1}, it follows that 
\begin{multline*}
\Tr_{L_j((u_j))/K(s))}\Wel_{L_j((u_j))}(f_{u_j}((C_{L_j((u_j))}))\\=\Tr_{L_j((s))/K((s))}(\Tr_{L_j((u_j))/L_j((s))}\Wel_{L_j((u_j))}(f_{u_j}((C_{L_j((u_j))}))))
\end{multline*}
extends to a section of $\sGW$ over $\Spec K[[s]]$. As 
\[
\Tr_{\prod_jL_j((u_j))/K((s))}=\sum_j\Tr_{L_j((u_j))/K((s))}, 
\]
this shows that $\Wel_d$  is unramified at $\bar{\eta}_{cusp}$.
\end{proof}

\section{Curves with a tacnode}

\begin{proposition} \label{prop:UnramTac} Let  $\bar{\eta}_{tac}=\mathfrak{z}(\eta_{tac})$ . Then $\Wel_d\in H^0(\Sym^{3d-1}(\P^2)^0_\delta, \sGW)$ is unramified at $\bar{\eta}_{tac}$.
\end{proposition}

\begin{proof} Let $u$ be a parameter at $\eta_{tac}$ in $\bar{\sM}_{0,3d-1}^\Sigma(\P^2, d)$ and let $K=k(\eta_{tac})$. We write $f:C\to \P^2_K$ for $f_{\eta_{tac}}:C_{\eta_{tac}}\to \P^2_{k(\eta_{tac})}$.

Let $f_u:C_u\to \P^2_{K((u))}$ be the morphism corresponding to the morphism 
\[
\Spec K((u))\to \bar{\sM}_{0,3d-1}^\Sigma(\P^2, d). 
\]
Since the map $f$ is unramified, the map $\mathfrak{z}$ is \'etale at $\eta_{tac}$, and we just need to see that $\Wel_{K((u))}(f_u(C_u))\in \GW(K((u)))$ is unramified at $u$. 

Since $f(C)$ has a single tacnode, this tacnode is a $K$-point of $f(C)$, so we may assume that $f(C)$ has $q_0:=(1:0,0)$ as its tacnode. The closed subscheme $f^{-1}(q_0)$ is either a single closed point with residue field $K(f^{-1}(q_0))$ a degree two extension of $K$ or is a  disjoint union of two $K$-points of $C$. We discuss first the case of a single closed point $p\in C$. 

As the characteristic is not 2, we may write $K(p)=K[t]/(t^2-a)$ for some $a\in K^\times$. This gives us an isomorphism of the completion $\sO_{C,p}^\wedge$ with the completion 
\[
K[t]_{(t^2-a)}^\wedge:=\lim_{\substack{\leftarrow\\n} }K[t]/(t^2-a)^n.
\]
In suitable analytic coordinates $(x, y)$ for the completion 
$\sO_{\P^2, q_0}^\wedge$ we may express $f$  as
\[
f=f(t)=(t^2-a, t(t^2-a)^2)
\]
so in $\Spec\sO_{\P^2, q_0}^\wedge=\Spec K[[x,y]]$ we have the equation $y^2=(x+a)x^4$ describing $f(C)^\wedge:=f(C)\times_{\P^2}\Spec K[[x,y]]$.

The deformation theory of $f$ gives, for a suitable choice of parameter $u$, the description of $f_u:C_u^\wedge\to \P^{2\wedge}_{K[[u]]}$ as
\[
f_u(t)=(t^2-a, t(t^2-a)^2-ut).
\]
The image curve $f_u(C_u)^\wedge$ thus satisfies the equation $y^2=(x+a)(x^2-u)^2$. Over $K((u))$, this curve has the ordinary double point $p_u$ defined by the ideal $(x^2-u, y)\subset K((u))[[x, y]]_u$, where $K((u))[[x, y]]_u$ is the completion of $K((u))[x,y]$ at this ideal. $K((u))[[x, y]]_u$ has has parameters the elements $s:=x^2-u$ and $y$, and  $f_u(C_u)^\wedge$ thus satisfies the equation $y^2-(x+a)s^2=0$. The local invariant is thus $\<-x-a\>\in \GW(K((u))(p_u))$. 

The ``remaining'' singular points of $f_u(C_u)$ extend to define a finite \'etale cover of $\Spec K[[u]]$, so by Lemma~\ref{lem:WelschExt}, there is a unit $U\in K[[u]]^\times$ such that
\[
\Wel_{K((u))}(f_u(C_u))=\<U\cdot \Nm_{K((u))(p_u)/K((u))}(x+a)\>
\]
But $\Nm_{K((u))(p_u)/K((u))}-x-a=(x+a)(-x+a)=a^2-x^2=a^2-u$. Since $a\neq0$, this is a unit in $K[[u]]$, so $\Wel_{K((u))}(f_u(C_u))$ extends to a section of $\sGW$ over $K[[u]]$, proving the proposition.

The case of $f^{-1}(q_0)$ splitting into two $K$-points is similar but a bit easier. Since in this case $C$ has a $K$-point, $C$ is isomorphic to $\P^1$ and we may assume that $f^{-1}(q_0)=\{\pm 1\}$. The computation is then completely parallel, with the simplification that  $a=1$. 
\end{proof}

\section{Curves with a triple point}

\begin{proposition} \label{prop:UnramTri} Let  $\bar{\eta}_{tri}=\mathfrak{z}(\eta_{tri})$ . Then $\Wel_d\in H^0(\Sym^{3d-1}(\P^2)^0_\delta, \sGW)$ is unramified at $\bar{\eta}_{tri}$.
\end{proposition}

\begin{proof} The proof is analogous to the proof of Proposition~\ref{prop:UnramTac}. Let $u$ be a parameter at $\eta_{tri}$ in $\bar{\sM}_{0,3d-1}^\Sigma(\P^2, d)$ and let $K=k(\eta_{tri})$. We write $f:C\to \P^2_K$ for $f_{\eta_{tri}}:C_{\eta_{tri}}\to \P^2_{k(\eta_{tri})}$.

Let $f_u:C_u\to \P^2_{K((u))}$ be the morphism corresponding to the morphism 
\[
\Spec K((u))\to \bar{\sM}_{0,3d-1}^\Sigma(\P^2, d). 
\]
Since the map $f$ is unramified, the map $\mathfrak{z}$ is \'etale at $\eta_{tri}$, and we just need to see that $\Wel_{K((u))}(f_u(C_u))\in \GW(K((u)))$ is unramified at $u$. 

As before, we may assume that the triple point of $f(C)$ is at $q_0=(1:0:0)$. Then $f^{-1}(q_0)$ is a degree three extension of $K$. Since base extension by an odd degree extension is injective on $\GW$, we may assume that $f^{-1}(q_0)$ contains a $K$-point $p_1$. Thus $C\cong \P^1_K$. The remainder $f^{-1}(q_0)\setminus\{p_0\}$ is then a degree two extension of $K$; we consider first the case in which 
$f^{-1}(q_0)\setminus\{p_0\}$ is a single closed point $p_1$ of $\P^1$ with residue field $K(p_1)$ a degree two extension of $K$. 

We may assume that $p_0$ and $p_1$ lie in the affine open $T_0\neq0$. Letting $t=T_1/T_0$ be the standard coordinate, we may assume that $p_1$ is defined by the ideal $(t^2-a)$ for some $a\in K^\times$ and that  $p_0$ is defined by $t-b$ for some $b\in K$. We let $K[[t]^\wedge$ denote the completion of $K[t]$ with respect to the ideal $(t-b)(t^2-a)$. Then $K[[t]^\wedge$ breaks up as a product
\[
K[[t]^\wedge\cong K[[z]]\times K[t]/(t^2-a)[[w]]
\]
with $K[[z]]$ the completion at $t-b$ and $K(p_1)[[w]]$ the completion at $t^2-a$. For a suitable choice of parameters $z, w$ and analytic coordinates $x, y$ for $\sO_{\P^2, q_0}^\wedge$, the map $f$ is given by 
\[
f(z)=(z, 0),\ f(w)=(w, tw)
\]
and then $f(C)$ satisfies the equation $y(y^2-ax^2)=0$. 

As a triple point on a plane curve lowers the genus by 3, the curves in the deformation space of $f(C)$ corresponding to deformations $f_u:\P^1_{K[[u]]}\to \P^2_{K[[u]]}$ will have 3 nearby ordinary double points. The local deformation space for the equation $g(x, y)=y(y^2-ax^2)=0$ is the ring
\[
K[x, y]/(\del g/\del x, \del g/\del y)= K[x, y]/(3y^2-ax^2, xy)=K\cdot 1\oplus K\cdot x\oplus K\cdot y\oplus K\cdot y^2
\]
so the general deformation of $g=0$ is $g_v=0$, 
\[
g_v(x, y)=y(y^2-ax^2)+v_1+v_2x+v_3y+v_4y^2
\]
The $v$ for which $g_v=0$ has three double points will all give cubic plane curves with three double points, which must necessarily be three lines intersecting pairwise transversely.  An equation of the form $g_v(x, y)=0$ gives three lines if and only if 
\[
g_v(x, y)=g_u(x,y):=y(y^2-ax^2)+2uy^2 +u^2y=y((y+u)^2-ax^2)
\]
and the lines intersect transversely if and only if $u\neq0$.

As the map from the deformation space of $f$ to the deformation space of $f(C)\subset \P^2$ as a rational curve is \'etale, we may thus use $u$ as our local parameter at $\eta_{tri}$. The curve $g_u=0$ has the double points $q_{0u}=(0,-u)$ and $q_{1u}=(x_u, 0)$, with $x_u$ defined by the ideal $(u^2-ax^2)\subset K((u))[x]$. At $q_{0u}$, the local invariant is $-a$, at $q_{1u}$ we may use parameters $y$ and $(y+u)^2-ax^2$, so the local invariant is $-1$. As in the proof of Proposition~\ref{prop:UnramTac}, there is a unit $U\in K[[x]]^\times$ so that 
\[
\Wel_{K((u))}(f_u(C_u))=\<U\cdot -a\cdot \Nm_{K((u))(x_u)/K((u))}(-1)\>=
\<-aU\>
\]
so $\Wel_{K((u))}(f_u(C_u))$ is unramified at $\eta_{tac}$.

If $f^{-1}(q_0)\setminus\{p_0\}$ consists of two $K$-points, the proof is the same: we may take $a=1$.
\end{proof}

\section{Reducible curves}

\begin{proposition} \label{prop:UnramRed} Let  $\bar{\eta}_{d_1, n_1}=\mathfrak{z}(\eta_{d_1, n_1})$ . Then $\Wel_d\in H^0(\Sym^{3d-1}(\P^2)^0_\delta, \sGW)$ is unramified at $\bar{\eta}_{d_1, n_1}$.
\end{proposition}

\begin{proof} Let $K=k(\eta_{d_1, n_1})$ and write $f_{\eta_{d_1, n_1}}:C_{1, \eta_{d_1, n_1}}\cup C_{2, \eta_{d_1, n_1}}\to \P^2_K$ as $f:C_1\cup C_2\to\P^2_K$.  $f(C_1)$ is of degree $d_1$ and $f(C_2)$ of degree $d_2=d-d_1$ and it follows by a dimension count that, writing $\mathfrak{d}=\mathfrak{d}_1+\mathfrak{d}_2$ with $\mathfrak{d}_i$ supported on $C_i$ and of degree $n_i$, we must have $n_1=3d_1-1$ and $n_2=3d_2$ or 
$n_1=3d_1$ and $n_2=3d_2-1$. In particular, exactly one of $\mathfrak{d}_1, \mathfrak{d}_2$ has degree congruent to 2 modulo 3, and thus there are no automorphisms of the map $f$.  Therefore $\bar{\sM}_{0,3d-1}^\Sigma(\P^2, d)$ is a scheme in a neighborhood of $\eta_{d_1, n_1}$, the map $\mathfrak{z}$ is unramified at $(f, C, \mathfrak{d})$, and we need only check, that for a parameter $u$ at $\eta_{d_1, n_1}$ in $\bar{\sM}_{0,3d-1}^\Sigma(\P^2, d)$, with corresponding morphism $f_u:C_u\to \P^2_{K[[u]]}$, the  invariant $\Wel_{K((u))}(f_u(C_u))$ is unramified at $u$. 

Let $\Sing(f_u(C_u))\subset f_u(C_u)$ be the scheme of singularities of the fibers. We have the isolated point  $q_0:=f(C_1\cap C_2)$ of $\Sing(f_u(C_u))$, but the open and closed subscheme $\Sing(f_u(C_u))\setminus \{q_0\}$ is finite and \'etale over $\Spec K[[u]]$. It follows from  Lemma~\ref{lem:WelschExt} that $\Wel_{K((u))}(f_u(C_u))$ is unramified at $u$. 
 \end{proof}
 
\section{Del Pezzo surfaces} We have concentrated on the case of $\P^2$, but as  in the paper of Itenberg-Kharlamov-Shustin \cite{IKS}, the methods described above work as well, with some appropriate restrictions, for a smooth del Pezzo surface $S$.  Recall that a del Pezzo surface over a field $k$ is a smooth projective surface $S$ over $k$ such that the anti-canonical class $-K_S$  is ample.  The definition of the Welschinger invariant of a geometrically integral nodal curve $C\subset S$ is exactly the same as for $S=\P^2$, or for an arbitrary smooth projective surface over a field $K$, for that matter.  

We recall that the {\em degree} of a del Pezzo surface $S$ is $d(S):=K_S^2$. Over an algebraically closed field, a del Pezzo surface is either a relative minimal model of $\P^2$ or the blow up of $\P^2$ at at most 8 points, that is, $S$ fits into a tower 
\[
S=S_0\to S_1\to \ldots \to S_d=\P^2
\]
with $0\le d\le 8$ and each $S_i$ is the blow-up of $S_{i+1}$ a some point $p_{d-i}$. The relatively minimal models are in three families: either $\P^2$, the family of Hirzebruch surfaces $F_n$ for $n$ even and the family of Hirzebruch surfaces $\Sigma_n$ for $n>1$ odd. $F_0=\P^1\times \P^1$, $F_1$ is the blow-up of $\P^2$ at one point (so not relatively minimal) and each $F_n$ has a small deformation that is either $F_1$ ($n$ odd) or $\P^1\times \P^1$. The relatively minimal models except for $\P^2$ all have degree 8, $\P^2$ has degree 9.

In particular, each del Pezzo surface  admits a small deformation to  a surface $S$ that, over an algebraically closed field is either  $\P^1\times\P^1$ or the blow-up of $\P^2$ at at most 8 (possibly 0) distinct points. To keep the discussion uniform, we call a surface that  over an algebraically closed field is a blow-up of $\P^2$ at at most 8 distinct points a {\em typical}   del Pezzo, so every del Pezzo has a small deformation to a typical del Pezzo or a form of $\P^1\times\P^1$. 

Let $S$ be a del Pezzo over a field $K$ isomorphic to $\P^1\times\P^1$ over $\bar{K}$. The two rulings $\pi_1:\P^1\times\P^1\to \P^1$ $\pi_2:\P^1\times\P^1\to \P^1$ then define two families over curves on $S\times_KL$ for some degree two extension $L$ of $K$ and thus give morphisms $p_1:S_L\to C_1$, $p_2:S_L\to C_2$ for $C_{1L}, C_{2L}$ conics over $L$. This  gives the isomorphism $(p_1, p_2)_L:S_L\to C_1\times_LC_2$.  The conjugation of $L$ over $K$ either exchanges the two rulings or not; if not, then the two families of curves on $S$ are defined over $L$ and we may descend $(p_1, p_2)_L$ to an isomorphism $(p_1, p_2):S\to C_1\times_K C_2$ for conics $C_1, C_2$ defined over $K$. If the conjugation exchanges the rulings, then $C_{1L}$ and $C_{2L}$ are isomorphic, $C_{1L}\cong C_L\cong C_{2L}$ for some conic $C$ defined over $K$ and we have the isomorphism $(p_1, p_2)_L:S_L\to C_L\times_LC_L$. Thus the diagonal on $C_L\times_LC_L$ gives a curve $\Delta_L\subset S_L$. If $\Delta_L$ is invariant under the conjugation, then $\Delta_L$ descends to a curve $\Delta\subset S$, and $\sO_S(\Delta)$ gives an embedding of $S$ as a quadric in $\P^3$. If not, then $\Delta_L\cap \Delta_L^\sigma$ defines a closed point $p$ of $S$ of degree 2 over $K$. 

In the first case, $S\cong C_1\times C_2$, $S$ has the degree 4 point $-p_1^*K_{C_1}\cap -p_2^*K_{C_2}$. If $S$ is a quadric in $\P^3$, then $S$ has a degree 2 point and in the last case, $S$ has a degree 2 point as well.  Thus a form of $\P^1\times\P^1$ always has a closed point of degree at most four over $K$; blowing up this point gives a typical del Pezzo surface.

If we have  a del Pezzo $S'$ over $K$ that has a small deformation $S$ (say over $K[[u]]$) with geometric model $\P^1\times\P^1$, then we may specialize the point on $S$ of degree $\le 4$ to $S'$; it is not hard to see that this can be done so that the resulting closed subscheme of the deformation family is \'etale over $K[[u]]$. Blowing up this point gives us a del Pezzo surface $\tilde{S}$ with a small deformation that is typical, and all our results about Welschinger  invariants for $S$ can be recovered from corresponding statements for $\tilde{S}$.

Fix a  an effective Cartier divisor $D$ on our del Pezzo surface $S$ with  $D^2\ge-1$. We consider the moduli space $\bar{\sM}^{*\Sigma}_{0,n}(S,D)$ parametrizing tuples $(C, f:C\to S, \mathfrak{d})$ with $C$ a genus zero quasi-stable curve, $\mathfrak{d}$ a reduced 0-cycle of degree $n$ on $C$ and $f:(C,\mathfrak{d})\to S$ a stable map, that is, over an algebraically closed field, if we write $\mathfrak{d}=\sum_{i=1}^np_i$, then $f:(C, (p_1,\ldots, p_n))\to S$ is stable.  In addition, we require that $f_*([C])$ is in the linear system $|D|$.

We have the open substack $\sM^\Sigma_{0,n}(S,D)$ of those $C$ which are geometrically integral and let $\bar{\sM}^\Sigma_{0,n}(S,D)$ be the closure of $\sM^\Sigma_{0,n}(S,D)$, that is, the unique connected component of $\bar{\sM}^{*\Sigma}_{0,n}(S,D)$ containing $\sM^\Sigma_{0,n}(S,D)$; this is a smooth Artin stack by the results of \cite{AbramOort, J}.

We replace $n=3d-1$ with $n=-D\cdot K_S-1:=r$ and replace $\Sym^n(\P^2)^0$ with $\Sym^n(S)^0$, the open subscheme of $\Sym^n(S)$ parametrizing reduced degree $n$ 0-cycles on $S$. We replace $\P^{N_d}$ with $|D|= \P^{N_D}$, and $P_{0,d}$ with $P_{0,D}$, the locally closed subset of  $\P^{N_D}$ of reduced irreducible curves $C\in |D|$ whose geometric normalization is a $\P^1$ and let $\bar{P}_{0,D}$ be the closure of $P_{0,D}$. We let $\sC_{0,D}\subset \bar{P}_{0,D}\times S$ be the universal effective Cartier divisor over $\bar{P}_{0,D}$.

We let $\bar{\sM}_{0,n}^\Sigma(S, D)^0$ be the inverse image of $\Sym^n(S)^0$ under the projection $\bar{\sM}_{0,n}^\Sigma(S, D)\to \Sym^n(S)$ and let $\sP^\Sigma_{S0}\to 
\bar{\sM}_{0,n}^\Sigma(S, D)^0$ be the universal curve. 

We replace the diagram \eqref{eqn:BasicDiagram} with 
\begin{equation}\label{eqn:BasicDiagram2}
\xymatrix{
&\sP^\Sigma_{S0}\ar[rd]^{f_{S0}^\Sigma}\ar[d]\\
&\bar{\sM}_{0,n}^\Sigma(S, D)^0\ar[dr]^{\mathfrak{D}_S^0}\ar[dl]_{\mathfrak{z}_S}&\sC_{0,D}\ar[d]\ar@{^(->}[r]&\bar{P}_{0,D}\times S\ar[dl]^{p_1}\\
\Sym^n(S)^0&&\bar{P}_{0,D}
}
\end{equation}

We define the substacks $\sM_{0,n}^\Sigma(S,D)^0$, $\bar{\sM}_{0,n}^\Sigma(S,D)_\delta$, $\sM_{0,n}^\Sigma(S,D)_\delta$, $\bar{\sM}_{0,n}^\Sigma(S,D)^0_\delta$, $\sM_{0,n}^\Sigma(S,D)^0_\delta$, $D^{(i)}(S,D)$, $D^{(i)}(S,D)_\delta$, $\bar{\sM}_{0,n}^\Sigma(S,D)_{?}$, $\sM_{0,n}^\Sigma(S,D)_{?}$ for $?\in \{cusp, tac, tri\}$ all as for $\P^2$.

For $D=D_1+D_2$, with $D_i$ both effective divisors, and with $r_1, r_2\ge 0$ integers with $r_1+r_2=r$,  we let $\bar{\sM}_{0,n}^\Sigma(S,D)_{(n_1, D_1)}$ be the closure in $\bar{\sM}_{0,n}^\Sigma(S,D)$ of the space of two-component reducible quasi-stable curves $C=C_1\cup C_2$ and 0-cycle $\mathfrak{d}=\mathfrak{d}_1+\mathfrak{d}_2$ with $f=f_1\cup f_2:C\to S$ such that $(f_i:C_i\to S, \mathfrak{d}_i)$ is in  $\bar{\sM}_{0,n_i}^\Sigma(S,D_i)$, $i=1, 2$. We define  $\bar{\sM}_{0,n}^\Sigma(\S,D)_{ord}$ as for $\P^2$:
\begin{multline*}
\bar{\sM}_{0,n}^\Sigma(S, D)_{ord}\\
=
\sM_{0,n}^\Sigma(S, D)^0_\delta\cup \sM_{0,n}^\Sigma(S, D)^0_{cusp}\cup \sM_{0,n}^\Sigma(S, D)^0_{tac,\delta}\\\cup \sM_{0,n}^\Sigma(S, D)^0_{tri,\delta}\cup (D^{(2)}(S,D)_\delta\setminus D^{(3)(S,D)})
\end{multline*}
We set  $\Sym^r(S,D)^0_{dgn}:=\mathfrak{z}(\bar{\sM}_{0,n}^\Sigma(S, D)^0\setminus \sM_{0,n}^\Sigma(\S, D)_\delta^0)$ and $\Sym^r(S,D)^0_\delta:=\Sym^r(S)^0\setminus
\Sym^r(S,D)^0_{dgn}$.
 
With these changes, the arguments of \cite[Lemmas 9, 10,11]{IKS} go through (even in positive characteristic $p>3$) to prove the analog of Lemma~\ref{lem:DimOrd}, with the following changes:\\[5pt]
1. $\bar{\sM}_{0,n}^\Sigma(S,D)$ might not be irreducible\\[3pt]
2. The statement on the codimension in (3) and (4) should be taken only for $n=r$, and should  refer to the image in $\Sym^r(S)$ of each generic point of each irreducible component $F$
\\[5pt]
These changes do not affect the overall argument and yield the analog of our Lemma~\ref{lem:DeltaProps}.

For $\mathfrak{d}\in \Sym^r(S,D)^0_\delta(K)$, we define $\Wel_{S,D}(\mathfrak{d})\in \GW(K)$ just as we did $\Wel_d(\mathfrak{d})$ for  $\mathfrak{d}\in \Sym^{3d-1}(\P^2)^0_\delta(K)$ and define $\Wel_{S,D}\in \GW(k(\Sym^r(S,D)^0))$ as $\Wel_{S,D}(\eta)$ for $\eta$ the generic point of $\Sym^r(S,D)^0$. We have the analog of Lemma~\ref{lem:Ext1}, namely

\begin{lemma}\label{lem:Ext1S}
$\Wel_{S,D}\in \GW(k(\Sym^r(S,D)^0))$ extends to a global section
\[
\Wel_d\in H^0(\Sym^r(S,D)^0_\delta, \sGW)
\]
Moreover
\[
ev_\mathfrak{d}(\Wel_{S,D})=\Wel_{S,D}(\mathfrak{d})
\]
for all  $\mathfrak{d}\in \Sym^r(S,D)^0_\delta$.
\end{lemma}
This does not require that $S$ be typical.

The consideration of the reducible curves requires some care. Here we have the following result taken from \cite{IKS}.

\begin{lemma} Suppose that $S_{\bar k}$ is the blow up of $s\le 8$ points in $\P^2_{\bar k}$, $p_1,\ldots, p_s$. We suppose that either $s<8$ or $s=8$ and $D\neq -2K_S$. Then for a general choice of $p_1,\ldots, p_s$,   $\mathfrak{z}(D^{(i)})$ has codimension $\ge 2$ on $\Sym^r(S)_0$ for all $i\ge 3$, and codimension $\ge1$ for $i=2$. Furthermore, $\bar{\sM}_{0,r}^\Sigma(S,D)$ is smooth at each generic point of $\bar{\sM}_{0,r}^\Sigma(S,D)_{(n_1, D_1)}$ and at  each generic point $\eta$ of  $\bar{\sM}_{0,r}^\Sigma(S,D)_{(n_1, D_1)}$ such that $\mathfrak{z}(\eta)$ has codimension one  on $\Sym^r(S)^0$, the corresponding maps $f_i:C_i\to S$ are unramified, $f(C_1\cup C_2)$ is reduced and has only ordinary double points. 
\end{lemma}

\begin{proof} This follows from \cite[Lemmas 9, 10,11]{IKS}. There is one exceptional case discussed in \cite[Lemma 11]{IKS}, namely, that of a del Pezzo of degree one ($s=8$) with $D=-2K_S$, that is, corresponding to the linear system of degree six curves in $\P^2$ having a double point at each $p_i$. In this case, it is possible that  $f(C_1)=f(C_2)$  are the same curve in $-K_S$.  
\end{proof}

We call a surface $S$ that satisfies the condition that $\mathfrak{z}(D^{(i)})$ has codimension $\ge 2$ on $\Sym^r(S)_0$ for all $i>2$ {\em general}. Clearly a small deformation of a typical del Pezzo is general. 

We discuss the exceptional case $s=8$, $D=-2K_S$. In this case the moduli space of stable maps of  a smooth integral genus 0 curve,  $\sM_{0,0}(S, -2K_S)$ is one dimensional, corresponding to maps $f:C\to \P^2$ with $f(C)$ having a point of multiplicity at least two at each $p_i$. Since $r=1$, there is no symmetric group action. We have the map $\mathfrak{z}: \bar{\sM}_{0,1}(S, -2K_S)\to S=\Sym^1(S)_0$ and for each $p\in S(K)$,  $\mathfrak{z}^{-1}(p)$ will contain a point $x_0$ corresponding to a map $f_0:C_1\cup C_2\to S$ with $f(C_1)=f(C_2)$, each corresponding to a nodal cubic curve in $\P^2$ containing $p$ and $p_1,\ldots, p_8$. As only one of $C_1, C_2$ can contain the 0-cycle $p$, there are no automorphisms of this data, and thus $\bar{\sM}_{0,1}(S, -2K_S)$ is a smooth scheme in a neighborhood of $x_0$. 

Let $q_1\neq q_1'\in C_1$ be the geometric points with $f_0(q_1)=f_0(q_1')$ the single node on the curve $f_0(C_1)$ and define $q_2\neq q_2'$ on $C_2$ similarly. For general $p_1,\ldots, p_8$, the points $q_i, q_i'$ will be disjoint from the intersection point $C_1\cap C_2$. Setting $K=k(x_0)$,  the local deformation theory of $f_0$ gives the neighborhood of $f_0$ in $\bar{\sM}_{0,1}(S, -2K_S)$ corresponding to a smoothing of $C_1\cup C_2$ to a smooth irreducible genus 0 curve $C_u$   and $f_0$ deforming to $f_u:C_u\to S$; we pass to the completion of $x_0$ in $\bar{\sM}_{0,1}(S, -2K_S)$ and consider $f_u$ as a map $f_u:C_u\to S_{K[[u]]}$. If the $p_1,\ldots, p_8$ are general, then $f_u(C_u)$ will be a curve on $S$ with two (geometric) nodes; let  $q_1(u), q_1'(u)$ be the (geometric) points of $(C_u)$ with $f_u(q_1(u))=f_u(q_1'(u)$ one of the nodes of $f_u(C_u)$ and let  $q_2(u), q_2'(u)$ be the (geometric) points of $(C_u)$ with $f_u(q_2(u))=f_u(q_2'(u)$ the other node.   The tuple $(q_1(u), q_1'(u), q_2(u), q_2'(u))$ specialize over $u\mapsto 0$ to $(q_1, q_1', q_2, q_2')\in C_1\cup C_2$. As $q_1, q_1'$ is in $C_1\setminus C_2$, and $C_1, C_2$ are not conjugate over $K$, we may order the  tuple $(q_1(u), q_1'(u), q_2(u), q_2'(u))$ so that $(q_1(u), q_1'(u))$ specializes to $(q_1, q_1')$ and $(q_2(u), q_2'(u))$ specializes to $(q_2, q_2')$. Thus the closed subscheme $Q(u)\subset C_u$ defined by $(q_1(u), q_1'(u), q_2(u), q_2'(u))$ is \'etale over $\Spec K[[u]]$.

We form the scheme $\bar{C}_u$ over $K[[u]]$ by gluing $(q_1(u), q_1'(u))$ and $(q_2(u), q_2'(u))$ to points $p_1(u)$, $p_2(u)$. The map $f_u:C_u\to S_{K[[u]]}$ factors through the projection $C_u\to \bar{C}_u$ via $\bar{f}_u:\bar{C}_u\to S_{K[[u]]}$. Moreover, 
$\bar{f}_u\otimes K((u)):\bar{C}_u\otimes K((u))\to S_{K((u))}$ defines an isomorphism of 
$\bar{C}_u\otimes K((u))$ with the generic image $f_u(C_u\otimes K((u)))$, and we may therefore use $\bar{C}_u\otimes K((u))$ to define the Welschinger invariant $\Wel_{K((u))}(f_u(C_u\otimes K((u))))$. Since $\bar{C}_u\otimes K$ has only ordinary double points, and  the closure of  singular locus of $\bar{C}_u\otimes K((u))$  is \'etale over $K[[u]]$, it follows from Lemma~\ref{lem:Etale1} and Lemma~\ref{lem:WelschExt} that $\Wel_{K((u))}(f_u(C_u\otimes K((u))))$ is unramified at $u=0$.

With this additional argument in the special case $s=8$, $D=-2K_S$, the proof of a Proposition~\ref{prop:UnramRed} goes through in the case of a general del Pezzo surface $S$. The arguments for proving  Proposition~\ref{prop:UnramCusp}, Proposition~\ref{prop:UnramTac}  and Proposition~\ref{prop:UnramTri} are purely local and go through without change for a general del Pezzo surface $S$. This yields the following theorem, the analog of Theorem~\ref{thm:WelGlobal}  for general del Pezzo surfaces.

\begin{theorem}\label{thm:WelGlobalGen} Let $S$ be a general del Pezzo surface, $D$ an effective divisor on $S$ with $D^2\ge-1$ and let $r=-D\cdot K_S-1$.Then\\[5pt]
 1. $\Wel_{S,D}\in \GW(k(\Sym^r(S)^0))$ extends to a global section
\[
\Wel_{S,D}\in H^0(\Sym^r(S)^0, \sGW)
\]
2. For $\mathfrak{d}\in \Sym^r(S)^0$, let 
\[
\ev_\mathfrak{d}:H^0(\Sym^r(S)^0, \sGW)\to \GW(k(\mathfrak{d}))
\]
be the evaluation map. Then 
\[
ev_\mathfrak{d}(\Wel_{S,D})=\Wel_{S,D}(\mathfrak{d})
\]
for all  $\mathfrak{d}\in \Sym^r(S)^0_\delta$.
\end{theorem}
The proof is the same as for Theorem~\ref{thm:WelGlobal} with the modifications described above. This easily yields the same result for all del Pezzo surfaces by a specialization argument.

\begin{corollary}\label{cor:WelGlobalDelPezzo} Let $S$ be a del Pezzo surface, $D$ an effective divisor on $S$ with $D^2\ge-1$ and let $r=-D\cdot K_S-1$.Then\\[5pt]
 1. $\Wel_{S,D}\in \GW(k(\Sym^r(S)^0))$ extends to a global section
\[
\Wel_{S,D}\in H^0(\Sym^r(S)^0, \sGW)
\]
2. For $\mathfrak{d}\in \Sym^r(S)^0$, let 
\[
\ev_\mathfrak{d}:H^0(\Sym^r(S)^0, \sGW)\to \GW(k(\mathfrak{d}))
\]
be the evaluation map. Then 
\[
ev_\mathfrak{d}(\Wel_{S,D})=\Wel_{S,D}(\mathfrak{d})
\]
for all  $\mathfrak{d}\in \Sym^r(S,D)^0_\delta$.
\end{corollary}

\begin{proof} We already have an extension of $\Wel_{S,D}$ to a global section of $\sGW$ over $ \Sym^r(S,D)^0_\delta$ satisfying the property (2), so we need only see that $\Wel_{S,D}$  extends to a global section over all of $\Sym^r(S)^0$. We need only show that $\Wel_{S,D}$ is unramified at all codimension one points of $\Sym^r(S)^0\setminus \Sym^r(S,D)^0_\delta$. 

As mentioned before, we reduce to the case of typical $S$ by blowing up $S$ at some point, $g:\tilde{S}\to S$ and replacing $D$ with $g^*(D)$. Since the blow-up leaves unaltered all codimension one points of $\Sym^r(S)^0$, we may replace $S$ with $\tilde{S}$.

For typical $S$, there is a smooth projective family $\sS\to \Spec k[[u]]$ with special fiber $S$ and general fiber a general del Pezzo $S_u=\sS_u\otimes k((u))$; we may also lift $D$ to an effective Cartier divisor $\sD$ on $\sS$; let $D_u:=\sD\otimes k((u))$. We have the diagram
\begin{equation}\label{eqn:BasicDiagram3}
\xymatrix{
&\sP^\Sigma_{\sS0}\ar[rd]^{f_{\sS0}^\Sigma}\ar[d]\\
&\bar{\sM}_{0,r}^\Sigma(\sS, \sD)^0\ar[dr]^{\mathfrak{D}_{\sS}^0}\ar[dl]_{\mathfrak{z}_{\sS}}&\sC_{0,\sD}\ar[d]\ar@{^(->}[r]&\bar{P}_{0,\sD}\times \sS\ar[dl]^{p_1}\\
\Sym^r_{k[[u]]}(\sS)^0&&\bar{P}_{0,\sD}
}
\end{equation}
over $k[[u]]$ with special fiber the diagram \eqref{eqn:BasicDiagram2} for $(S,D)$ and with general fiber the diagram \eqref{eqn:BasicDiagram2} for $(S_u, D_u)$.

Let $\tilde{\eta}=\Spec \sO_{\Sym^r_{k[[u]]}(\sS), \eta}$. Then $\tilde{\eta}$ is a flat $k[[u]]$-scheme with generic fiber a generic point of $\Sym^r_{k((u))}(S_u)$ and special fiber $\eta$. The morphism $\mathfrak{z}_{\sS}$ is \'etale and proper over  $\tilde{\eta}$ and each $\tilde{x}_i\in \mathfrak{z}_{\sS}^{-1}(\tilde{\eta})$ gives a family of curves $f_{\tilde{\eta}}(C_{\tilde{\eta}})$ in $\sS_{\tilde{\eta}}$ with only ordinary double points in the fiber over $k((u))$ and over $k$. As the number of double points remains constant, the scheme of singular points of this family is \'etale over $\tilde{\eta}$ (Lemma~\ref{lem:Etale1}). Using Lemma~\ref{lem:WelschExt}, this tells us that the Welschinger invariant $\Wel_{S_u, D_u}(\eta_u)\in \GW(k((u))(\Sym^r(S_u)^0)$ is unramified at the codimension one point 
$\eta$ of $\Sym^r_{k[[u]]}(\sS)^0$ and thus  $\Wel_{S_u, D_u}(\eta_u)$ extends to a section $\Wel_{\sS, \sD}$ of $\sGW$ over some open subscheme $\sU$ of $\Sym^r_{k[[u]]}(\sS)^0$ containing $\eta$ and containing $\Sym^r_{k((u))}(S_u)^0$ and with 
$\Wel_{\sS, \sD}(\eta)=\Wel_{S, D}(\eta)$.  Then $\sU$ contains all codimension one points of $\Sym^r_{k[[u]]}(\sS)^0$ hence $\Wel_{\sS, \sD}$ extends to a global section of $\sGW$ over $\Sym^r_{k[[u]]}(\sS)^0$. In particular, the restriction of $\Wel_{\sS, \sD}$ to $\Sym^r(S)^0$ defines  a global section of $\sGW$ with value $\Wel_{S, D}(\eta)$ at $\eta$. Thus $\Wel_{S, D}(\eta)$ is everywhere  unramified and extends to a global section of $\sGW$ over all of $\Sym^r(S)^0$.
\end{proof}

\begin{corollary} Let $D$ be an effective divisor on a del Pezzo surface $S$ with $D^2\ge-1$, with $D$ and $S$ defined over a field $K$ of characteristic $>3$ and let $r=-D\cdot K_S-1$. Then for $\mathfrak{d}_1, \mathfrak{d}_2$ in $\Sym ^r(S,D)^0_\delta(K)$,  in the same $\A^1$-connected component of $\Sym ^r(S)^0(K)$, we have
\[
\Wel_{S,D}(\mathfrak{d}_1)=\Wel_{S,D}(\mathfrak{d}_2).
\]
\end{corollary}

\begin{remark} \label{rem:WelComp3} Let $p_a(D)$ be the arithmetic genus
\[
p_a(D):=\frac{1}{2}D\cdot (D+K_S) +1.
\]
Just as for curves in $\P^2$ (Remark~\ref{rem:WelComp2}), for  $\mathfrak{d}$ an $\R$-point of $\Sym^r(S)^0$, $\Wel_{S,D}(\mathfrak{d})\in \GW(\R)$ has signature $(-1)^{p_a(D)}$ times Welschinger's invariant for the signed count of real degree $d$ curves in $S$ containing $\mathfrak{d}$ and the rank of $\Wel_{S,D}(\mathfrak{d})$ is just the number of curves (over $\C$) containing $\mathfrak{d}$.
\end{remark}

\end{document}